\documentclass[ronefignum,onetabnum]{siamart171218}

\usepackage{lipsum}
\usepackage{amsfonts}
\usepackage{graphicx}
\usepackage{epstopdf}
\usepackage{algorithmic}
\usepackage{amsmath,amssymb,mathrsfs}
\usepackage{color}
\usepackage{latexsym}
\usepackage{bm,bbm}

\ifpdf
  \DeclareGraphicsExtensions{.eps,.pdf,.png,.jpg}
\else
  \DeclareGraphicsExtensions{.eps}
\fi

\newcommand{\Prob}[1]{\mathbb{P}\left[#1\right]}

\newcommand{\Exp}[1]{\mathbb{E}\left[#1\right]}
\newcommand{\ExpPN}[1]{\mathbb{E}^0_N\left[#1\right]}

\newcommand{\Expi}[1]{\mathbb{E}^0_i\left[#1\right]}
\newcommand{\Expj}[1]{\mathbb{E}^0_j\left[#1\right]}

\newsiamremark{remark}{Remark}
\newsiamremark{hypothesis}{Hypothesis}
\crefname{hypothesis}{Hypothesis}{Hypotheses}
\newsiamthm{claim}{Claim}

\headers{RMF neural networks}{F. Baccelli and T. Taillefumier}

\title{Replica-mean-field limits \\ for intensity-based neural networks}

\author{Fran\c{c}ois Baccelli \thanks{Department of Mathematics and Department of Electrical and Computer Engineering, University of Texas, Austin, TX 
  (\email{francois.baccelli@austin.utexas.edu}).}
\and Thibaud Taillefumier \thanks{Department of Mathematics and Department of Neuroscience, University of Texas, Austin, TX 
  (\email{ttaillef@austin.utexas.edu}).}
  }

\usepackage{amsopn}

\ifpdf
\hypersetup{
  pdftitle={Replica-mean-field limits  for intensity-based neural networks},
  pdfauthor={F. Baccelli and T. O. Taillefumier}
}
\fi

\begin{document}

\maketitle

\begin{abstract}
Neural computations emerge from myriads of neuronal interactions occurring in intricate spiking networks.
Due to the inherent complexity of neural models, relating the spiking activity of a network to its structure requires simplifying assumptions, such as considering models in the thermodynamic mean-field limit.
In the thermodynamic mean-field limit, an infinite number of neurons interact via vanishingly small interactions, thereby erasing the finite size of interactions.
To better capture the finite-size effects of interactions, we propose to analyze the activity of neural networks in the replica-mean-field limit. 
Replica-mean-field models are made of infinitely many replicas which interact according to the same basic structure as that of the finite network of interest.
Here, we analytically characterize the stationary dynamics of an intensity-based neural network with spiking reset and heterogeneous excitatory synapses in the replica-mean-field limit. 
Specifically, we functionally characterize the stationary dynamics of these limit networks via ordinary differential equations derived from the Poisson Hypothesis of queuing theory.
We then reduce this functional characterization to a system of self-consistency equations specifying the stationary neuronal firing rates.
Of general applicability, our approach combines rate-conservation principles from point-process theory and analytical considerations from generating-function methods.
We validate our approach by demonstrating numerically that replica-mean-field models better capture the dynamics of neural networks with large, sparse connections than their thermodynamic counterparts.
Finally, we explain that improved performance by analyzing the neuronal rate-transfer functions, which saturate due to finite-size effects in the replica-mean-field limit.
\end{abstract}


%

\section{Introduction}
Intensity-based networks form a natural and flexible class of models for neural networks,
whose study has a long and successful history in computational neuroscience \cite{Bialek:1999,DayanAbbott,Emery:2004,Pillow:2008aa}.
In these models, the spiking activity of neural networks is represented in terms of point processes that are governed by neuronal ``stochastic intensities'' \cite{Daley1,Daley2}. 
Neuronal stochastic intensities model the instantaneous firing rate of a neuron as a function of the spiking inputs received from other neurons, thereby mediating network interactions and possibly carrying out local computations.
Detailed knowledge about intensity-based networks is mostly limited to simplifying limits such as the thermodynamic limit, i.e.,
with a very large number of neurons interacting very weakly \cite{Amari:1975aa,Amari:1977aa,Sompolinski:1988,Faugeras:2009}.  
Such an approximation, which neglects the finite-size of neuronal interactions, precludes explaining and predicting several aspects of neural computations,
including dynamical metastability \cite{Abeles:1995,Tognoli:2014aa},
correlation regime of activity \cite{Goris:2014aa,Lin:2015aa} and modulation of variability \cite{Churchland:2010aa,Ecker:2016}.
There is a crucial need for a computational framework allowing for the analysis of structured neural networks, while taking into account the finiteness of neuronal interactions.

Here, we develop a computational framework based on replica-mean-field (RMF) limits to address this need.
RMF limits were introduced in two distinct contexts: in statistical physics with applications to information-capacity
calculations in neuroscience \cite{Amit:1985,Mezard:1987aa,Gardner:1988}
and in computer networking to analyze communication networks in terms
of point processes \cite{Vvedenskaya:1996,RybShlosI,Benaim:121369,Baccelli:2018aa}.
We are concerned with the latter approach. 
Instead of considering the finite neural network of interest, this RMF approach considers 
closely related limit networks made of infinitely many replicas with the same basic neural structure.
The core motivation for considering RMF networks is that, under the assumption of independence between replicas,
referred to as the ``Poisson Hypothesis'' \cite{RybShlosI,RybShlosII}, these networks become analytically tractable.
In this work, we exploit the Poisson Hypothesis to characterize analytically the long-time
limit of a class of excitatory, intensity-based networks, called linear Galves-Loch\"erbach (LGL) models.
In considering this specific class of networks, our goal is to establish the foundation for the RMF
computational framework in a simple setting rather than aiming at generality.

For concreteness, let us introduce the RMF framework for a simple intensity-based network, 
namely the ``counting-neuron'' model. The counting-neuron model consists of a fully-connected network of $K$
exchangeable neurons with homogeneous synaptic weights $\mu$. 
For each neuron $i$, $1 \leq i \leq K$, the stochastic intensity $\lambda_i$ increases by $\mu>0$
upon reception of a spike and reset upon spiking to its base rate $b$.
Thus, its stochastic intensity is $\lambda_i(t)=b+\mu C_i(t)$, where $C_i(t)$ is
the number of spikes received at time $t$ since the last reset. 
Assuming the network state $\lbrace C_1(t), \dots , C_K(t) \rbrace$ has a well-defined stationary distribution,
a natural question is: how does the stationary firing rate $\beta= \Exp{\lambda_i}$ depend on $\mu$ and $K$? 
Strikingly, despite its simplicity, direct treatment of the model, e.g., via its master Kolmogorov equation, 
fails to yield an exact answer for non-degenerate $K$ and $\mu$ \cite{Ocker:2017aa}. 
This failure is primarily due to the presence of high-order correlations among subsets of neurons. 
Virtually all available results are obtained via a mean-field approximation in the thermodynamic limit, i.e., 
when letting $K \to \infty$ (large networks) and $\mu \to 0$ \cite{Baladron:2012aa} (vanishing interactions).
In this approximation, high-order correlations disappear at the cost of neglecting the finite-size
effects of neural constituents \cite{Touboul:2011}.

In principle, RMF models can incorporate correlations up to a given integer order $q$.
In this work, we only consider first-order replica models ($q=1$), which capture first-order statistics.
For the counting model and for an integer $M>0$, the $M$-replica model of first order consists of $M$ replicas,
each comprising $K$ counting neurons. Upon spiking, a neuron $i$ in replica $m$, indexed by $(i,m)$, delivers 
spikes with synaptic weight $\mu$ to the $K-1$ neurons $(j,v_j)$, $j\neq i$, where the replica destination $v_j$
is chosen uniformly at random. Thus, the probability for two replicas to interact over a finite period of time
vanishes in the limit $R \to \infty$, which intuitively justifies the Poisson Hypothesis.
The asymptotic independence between replicas makes a direct analytical treatment of the model possible,
just as in the traditional thermodynamic mean-field (TMF) limit. However, by contrast with the traditional
TMF limit, the stationary state explicitly depends on the finite-size parameters $K$ and $\mu$. 
We define the RMF model of the counting model as the replica network obtained in the limit of infinitely many replicas,
namely infinite $M$ but fixed and finite $K$.

The Poisson Hypothesis allows one to truncate the correlation terms due to neuronal interactions
in the functional characterization of the stationary state of a single replica.
For instance, in the counting neuron model, we will show that one can derive a single ordinary
differential equation (ODE) for $G$, the probability-generating function (PGF) of a neuron count $C$:
\begin{eqnarray}\label{eq:intro}
 \beta \!- \!\mu z G'(z) \!+\! \big(\beta (K\!-\!1)(z\!-\!1)\!-\!b \big)G(z) \!=\!0    \, .
\end{eqnarray} 
The truncation of the correlation terms comes at the cost of introducing the firing rate $\beta$
as an unknown parameter in \eqref{eq:intro}. As the ODE \eqref{eq:intro} is otherwise
analytically tractable, characterizing the RMF stationary state amounts to specifying the unknown
firing rate $\beta$. Then, the challenge of the RMF approach consists in specifying the unknown
firing rate via purely analytical considerations about a parametric system of ODEs.
For this model, we will show that in the RMF limit, the stationary firing rate $\beta$ is determined
as the unique solution of  
\begin{eqnarray}\label{eq:selfCons}
\beta = \frac{\mu c^a e^{-c}}{\gamma (a,c)} \quad \mathrm{with} \quad a = \frac{(K-1)\beta+b}{\mu} \quad \mathrm{and} \quad c = \frac{(K-1)\beta}{\mu} \, ,
\end{eqnarray}
where $\gamma$ denotes the lower incomplete Euler Gamma function.

Introduced for the counting-neuron model, the analytical strategy presented above is at the core of our general RMF approach.
In this work, we generalize this strategy to first-order replica networks with continuous
relaxation of the stochastic intensities and with general heterogeneous excitatory synaptic connections.
This generalization, which is stated in \Cref{th:relaxingNeuron}, is our main computational result.
While establishing this result, we develop a general methodology for the stationary analysis of RMF models, which we summarize below.
We also briefly discuss the relevance of applying the RMF limit to neural dynamics.\\


{\bf {Methodology.}} 
For clarity, we summarize the essential tenets of the methodology exposed herein.
Even under the simplest assumptions, there are no known analytical solutions to the Kolmogorov equations capturing the dynamics of intensity-based networks. 
Instead, one has to resort to analyzing caricatures of the dynamics
based on some simplifications of its correlation structure. 
The rate-conservation principle (RCP) of Palm calculus offers a systematic way of proposing such simplifications.
The Palm probability of a stationary point process can be interpreted as the distribution of
this point process conditioned to have a point present at the origin of the time axis.
The RCP consists in a conservation formula balancing the smooth drift of the stationary state
variables and their jumps at epochs of the point processes. In the RCP formula, jump terms
are expectations with respect to Palm probabilities, 
whereas the smooth dynamics leads to expectations with respect
to the stationary law of the system. Typically, the simplification then consists in replacing
these Palm expectations by stationary expectations, i.e., in ignoring the Palm bias. 
The resulting simplified functional equations can be solved in some fortunate cases. 
The key to such resolutions is to realize that our simplification of the RCP has a dynamical-system interpretation, which can be precisely formulated as a RMF limit. 
Indeed, the hallmark of RMF limit is to decouple network constituents, thereby cancelling out Palm biases.
This observation is instrumental in guaranteeing that there is at least one probabilistic, physical solution to our simplified functional equations.
Such solutions are found by imposing analyticity requirements that any probabilistic solution must satisfy. \\

{\bf {Applications.}} We do not intend to systematically investigate the applications 
of the RMF approach to neuroscience here, but rather, to highlight two key features of the RMF limit. 
First, we numerically simulate exemplars of recurrent and feedforward networks to compare the performance of RMF models and TMF models.
We show that TMF models outperform TMF models in predicting the neuronal spiking rates in LGL networks with strong, sparse synaptic interactions.
Second, we perform an asymptotic analysis of the neuronal rate-transfer functions, which are determined by the self-consistency equations in both the RMF and TMF limits.
This analysis shows that the RMF limit fundamentally differs from the classical TMF limit because increasing synaptic weights at fixed input rates leads to saturation---an aspect that cannot be captured by TMF models which consequently overestimate firing rates.
Further applications to neural-network analysis, such as higher-order models, are possible.
Beyond neuroscience, our methodology is also applicable to generic intensity-based stochastic network dynamics.
This suggests using the RMF framework to revisit classical problems in queuing theory, particle-system theory,
communication networks, population dynamics, epidemics, as well as 
completely new problems arising in, e.g., social network dynamics. \\

{\bf Structure.} 
In \Cref{sec:ppframework}, we introduce the point-process modeling framework for which we will develop RMF networks.
In \Cref{sec:pproach}, we characterize analytically the stationary state of RMF networks.
The neuroscience implications and the computational relevance of this approach are discussed in \Cref{sec:impli}.
\Cref{sec:proof} comprises the proofs supporting the results presented in \Cref{sec:ppframework} and in \Cref{sec:pproach}.
Future research directions are presented in Section \Cref{sec:future}, where we explain that
similar strategies apply for $i)$ any correlation orders $q$, namely with 
replica constituents being subsets of $q$ interacting neurons rather than single neurons, and $ii)$
for networks with heterogeneous synaptic weights supporting both excitation and inhibition. 
\\


{\bf Related work.} The inspiration for the replica models proposed in this work is rooted in the theory of nonlinear Markov processes, 
which were introduced by McKean \cite{McKean:1966}. These processes were extensively used to study mean-field
limits in queueing systems, initially by the Dobrushin school 
\cite{Vvedenskaya:1996,Rybko:2009aa,RybShlosI,RybShlosII}, and later by M. Bramson \cite{bramson:2011}.
This literature has two distinct components: 
$i)$ a probabilistic component proving asymptotic independence from the equations satisfied by the
non-linear Markov process, and $ii)$ a computational
component deriving closed-form expressions for the mean-field limit of the system of interest.
These two components jointly led to a wealth of new results in queueing theory, concerning
both open and closed queueing networks, e.g., \cite{Vvedenskaya:1996}. 
The aim of this work is to show that, just as in queueing theory, studying neural networks in the RMF
limit is computationally tractable.
Finding moment-generating functions by imposing condition of analyticity on some solutions is a classical approach in queueing theory \cite{Takacs:1962aa}.
The RCP simplification described in the methodology summary
were used to analyze point-process-based dynamics in peer-to-peer networks \cite{Baccelli:2017aa}
and in wireless networks \cite{SB17}.
However, the link established between RMF models and simplified RCP is novel.
Our approach also elaborates on prior attempts to solve the neural master equations in computational neuroscience.
Brunel {\it et al.} introduced mean-field limits for large neural networks with weak interactions 
from a computational perspective \cite{Abbott:1993,Brunel:1999aa, Brunel:2000aa}.
Touboul {\it et al.} then adapted the ideas of ``propagation of chaos'' for neural networks in the thermodynamic
mean-field limit \cite{Baladron:2012aa, Touboul:2012aa, Robert:2016}.
Their results were specialized to spiking models with memory resets by Galves and Locherb\"ach,
who also provided perfect algorithms to simulate the stationary states of infinite networks \cite{Galves:2013,DeMasi:2015}.
Closer to our approach, Buice, Cowan, and Chow adapted techniques from statistical physics to analyze
the hierarchy of moment equations obtained from the master equations \cite{Buice:2007,Buice:2009aa}.
These authors were able to truncate the hierarchy of moment equations to consider models
amenable to finite-size analysis via system-size or loop expansion around the mean-field solution \cite{Bressloff:2009aa}.
These authors also showed by field-theoretic arguments that the dynamics obtained by moment closure
was indeed that of a physical system.
Although the master equation of Buice {\it et al.} does not have a natural small parameter,
the moment-closure approach was implemented to solve the neural master equations in the thermodynamic limit \cite{Buice:2013}.
By contrast, our approach considers a new mean-field regime, that of the  RMF model for finite-size neural networks,
without any scaling of the interactions.\\


\section{Point-process framework for finite neural networks}\label{sec:ppframework}

In this section, we introduce the point-process modeling framework for which we will develop RMF networks.
In \Cref{sec:LGL}, we define the intensity-based neural networks that we consider throughout this work, i.e., 
linear Galves-L\"ocherbach (LGL) networks.
In \Cref{sec:MarkovAnalyis}, we justify that finite LGL networks admit a well-defined stationary regime 
with exponentially integrable neuronal stochastic intensities. In \Cref{sec:FunctionalAnalyis},
we derive functional equations characterizing the stationary joint distribution of the neuronal
stochastic intensities via the use of the RCP.  


\subsection{Linear Galves-L\"ocherbach models}\label{sec:LGL}

We consider a finite assembly of $K$ neurons whose spiking activity is modeled as the realization
of a system of simple point processes without common points $\bm{N} = \lbrace N_i \rbrace_{1\leq i \leq K}$
on $\mathbb{R}$ defined on an underlying measurable space $(\Omega, \mathcal{F})$.
For all neurons $1\leq i \leq K$, we denote by $\lbrace T_{i,n} \rbrace_{n \in \mathbb{Z}}$,
the sequence of successive spiking times with the convention that almost surely $T_{i,0} \leq 0 < T_{i,1}$ and $T_{i,n}<T_{i,n+1}$.
Each point process $N_i$ is a family $\lbrace N_i(B)\rbrace_{B \in \mathcal{B}(\mathbb{R})}$
of random variables with values in $\mathbb{N} \cup \lbrace \infty \rbrace$ 
indexed by the Borel $\sigma$-algebra $\mathcal{B}(\mathbb{R})$ of the real line $\mathbb{R}$.
Concretely, the random variable $N_i(B)$ counts the number of times that neuron $i$ spikes within 
the set $B$, i.e., $N_i(B) = \sum_{n \in \mathbb{Z}} \mathbbm{1}_B(T_{i,n})$. 
Setting the processes $N_i$, $1 \leq i \leq K$, to be independent Poisson processes
defines the simplest instance of our point-process framework as a collection of non-interacting neurons.

To model spike-triggered interactions within the network, we consider that the rate
of occurrences of future spikes depends on the spiking history of the network.
In other words, we allow the instantaneous firing rate of neuron $i$ to depend on the
times at which neuron $i$ and other neurons $j \neq i$ have spiked in the past.
Formally, the network spiking history $\lbrace \mathcal{F}_t \rbrace_{t \in \mathbb{R}}$
is defined as a non-decreasing family of $\sigma$-fields such that, for all $t$,
\begin{eqnarray}
 \mathcal{F}_t^{\bm{N}} = \left\{ \sigma\left( N_1(B_1), \ldots, N_K(B_K)\right) \, \vert \, B_i \in \mathcal{B}(\mathbb{R}) \, , \: B_i \subset (-\infty,t] \right\} \subset \mathcal{F}_t ,
\end{eqnarray}
where $\mathcal{F}_t^{\bm{N}}$ is the internal history of the spiking process $\bm{N}$.
The network spiking history $\lbrace \mathcal{F}_t \rbrace_{t \in \mathbb{R}}$ determines
the rate of occurrence of future spikes via the notion of stochastic intensity.
The stochastic intensity of neuron $i$, denoted by $\lbrace \lambda_i(t) \rbrace_{t \in \mathcal{R}}$,
can be seen as a function of $\lbrace \mathcal{F}_t \rbrace_{t \in \mathbb{R}}$
specifying the instantaneous firing rate of neuron $i$.
It is formally defined as the $\mathcal{F}_t$-predictable process $\lbrace \lambda_i(t) \rbrace_{t \in \mathcal{R}}$ satisfying 
\begin{eqnarray*}
\Exp{N_i(s,t] \, \vert \, \mathcal{F}_{s}}= \Exp{\int_s^t \lambda_i(s) \, ds\, \Big \vert \, \mathcal{F}_{s}} \, ,
\end{eqnarray*}
for all interval $(s,t]$ \cite{Jacod:1975}. 
Stochastic intensities generalize the notion of rate of events, or hazard function,
to account for generic history dependence beyond that of Poisson processes or renewal processes.

Specifying the history-dependence of the neuronal stochastic intensities entirely defines
a network model within the point-process framework. In this work, we consider models for which
the stochastic intensities $\lambda_1, \ldots, \lambda_K$ obey a system of coupled stochastic equations
\begin{eqnarray}\label{eq:dynmod}
\lefteqn{
\lambda_i(t) = \lambda_i(0) + \frac{1}{\tau_i} \int_0^t \left(b_i-\lambda_i(s) \right)\, ds +} \nonumber\\
&& \hspace{60pt} \sum_{j \neq i} \mu_{ij} \int_0^t N_j(ds) + \int_0^t \big(r_i-\lambda_i(s) \big) N_i(ds) \, ,
\end{eqnarray}
where the spiking processes $N_i$ have stochastic intensity $\lambda_i$. 
The above system of stochastic equations characterizes the history-dependence of the stochastic intensities.
The first integral term indicates that in between spiking events, $\lambda_i$ deterministically relaxes
toward its base rate $b_i>0$ with relaxation time $\tau_i$.
The second integral terms indicates that a spike from neuron $j \neq i$
causes $\lambda_i$ to jump by $\mu_{ij} \geq 0$, the strength of the synapse from $j$ to $i$.
Finally, the third integral term indicates that $\lambda_i$ resets to $0 \leq r_i \leq b_i$ upon spiking of neuron $i$. 
Taking $r_i < b_i$ models the refractory behavior of neurons whereby spike generation causes
the neuron to enter a transient quiescent phase. 

Thus-defined, our model can be seen as a system of coupled Hawkes processes with spike-triggered
memory reset and belongs to the Galves-L\"ocherbach class of models \cite{Galves:2013}.
Defining $T_{i,0}(t)$ to be the last spiking time before time $t$, i.e., 
$T_{i,0}(t) = T_{i,0} \circ \theta_t = \sup \lbrace s \leq t \vert N_s<N_t\rbrace$,
where $\theta_t$ is the time-shift operator, the stochastic intensity $\lambda_i(t)$ can be written under Galves-L\"ocherbach form 
\begin{eqnarray}
\lambda_i(t) = \phi_i\left( \sum_j \mu_{ij}\int_{T_{i,0}(t)}^t g_i(t-s) N_j(ds),t-T_0(t)\right)\, ,
\end{eqnarray}
with linear intensity functions $\phi_i$ and exponentially decaying kernels $g_i$:
\begin{eqnarray}
\phi_i(x,s)=x + b_i + (r_i-b_i)e^{-\frac{s}{\tau_i}} 
\quad \mathrm{and} \quad
g_i(t-s) = e^{-\frac{t-s}{\tau_i}} \, .
\end{eqnarray}
For this reason, we refer to our model as the linear Galves-L\"ocherbach model.
Galves-L\"ocherbach models have been primarily studied for infinite networks,
notably to characterize the mean-field dynamical limit \cite{DeMasi:2015,Delattre:2016}
or to construct perfect simulation algorithms \cite{Hodara:2017}.
Here, we focus on finite, excitatory assemblies of LGL neurons to approximate their 
dynamics via independent model akin to mean-field models but without taking any scaling limit.
That being said, we do not consider the proposed framework for its biological relevance
{\it per se} as we do not include important aspects of neural dynamics such as inhibition.
Our goal is rather to develop ideas amenable to generalization in a simple setting.


\subsection{Stationary Markovian dynamics} \label{sec:MarkovAnalyis}

In LGL networks, the stochastic intensity $\lambda_i(t)$ determines the instantaneous spiking rate
of neuron $i$ and can be viewed as the state of neuron $i$ at instant $t$.
When considered collectively, the stochastic intensities specify the network state
$\bm{\lambda}(t) = \left\{ \lambda_1(t), \ldots, \lambda_K(t)\right\}$ which follows
a continuous-time, pure-jump Markovian dynamics with infinitesimal generator   
\begin{eqnarray}\label{eq:infGen}
\mathcal{A}[f](\bm{\lambda}) = \sum_i \frac{b_i - \lambda_i}{\tau_i} \, \partial_{\lambda_i} f(\bm{\lambda})  + \sum_i \big( f(\bm{\lambda}+\bm{\mu}_i(\bm{\lambda})) - f(\bm{\lambda}) \big) \lambda_i \, ,
\end{eqnarray}
for all $f$ in $\mathcal{D}(\mathcal{A})$ the domain of $\mathcal{A}$.
In the above definition, the first sum collects the relaxation terms of the dynamics whereas the second sum corresponds to the interaction jumps triggered by the spiking of neuron $i$:
\begin{eqnarray}
\left[ \bm{\mu}_i (\bm{\lambda}) \right]_j = \left\{
\begin{array}{ccc}
\mu_{ji}  &  \mathrm{if} &  j \neq i \\
r_i - \lambda_i  &    \mathrm{if} &  j = i   
\end{array}
\right. \, .
\end{eqnarray}
Conditionally to the identity of the spiking neuron, the interaction jumps have fixed components set by the synaptic weights and a state-dependent component due to spiking reset.
The spiking reset to a history-independent state introduces a form of degeneracy which substantially hinders the analysis of the network dynamics, especially with respect to the regularity of the law of $\bm{\lambda}$.
In turn, for lack of a regularity characterization, it is unclear how to derive the Kolmogorov forward equation satisfied by $\bm{\lambda}(t)$ from the Kolmogorov backward equation $\partial_t u + \mathcal{A}[u] = 0$.

Despite these regularity complications, the stability of the network dynamics
can be established within the framework of Harris Markov chains \cite{MeynTweedie:2009}, 
whereby the continuous-time Markov chain $\{ {\bm \lambda}(t) \}_{t \in \mathbb{R}}$
proves to be Harris ergodic as long as $r_i>0$ for all neurons $i$. 
As the Markov chain $\{\bm  \lambda(t) \}_{t \in \mathbb{R}}$ is Harris ergodic (see the proof in \Cref{sec:Harris}), 
the network dynamics admits a unique invariant measure $p$ on $\mathbb{R}^K$ satisfying
\begin{eqnarray}\label{eq:invMeas}
\int_{\mathbb{R}^K} \mathcal{A}[f](\bm{\lambda}) \, p(d\bm{\lambda}) = 0 \; ,
\end{eqnarray}
for all $f$ in $\mathcal{D}(\mathcal{A})$. 
Sampling $\bm{\lambda}(0)$ according to the stationary measure $p$ defines
the stationary version of the Markov chain $\bm{\lambda}$, whose law $P$ is invariant under time shifts,
i.e., $P \circ \theta_t = P$ for all $t > 0$, and whose definition is naturally extended on the whole real line $\mathbb{R}$.
Coupling techniques using Nummelin splittings show that non-stationary dynamics converge
at least exponentially in total variation toward the stationary limit process \cite{Hodara:2018}.
The present work is only concerned with the stationary version of the network dynamics and, in the following,
the notation $\bm{\lambda}$ always refers to that stationary version.
Moreover, processes induced by $\bm{\lambda}$, such as the point processes $N_i$, inherit the stationary property.

We state the technical results justifying the existence of the stationary regime of the dynamics in \Cref{sec:proof}.
A key step is to check a Foster-Lyapunov drift condition in \Cref{th:FosterLyapunov} for the infinitesimal
generator $\mathcal{A}$ acting on exponential scale functions: $V_u({\bm \lambda}) = \exp{ \left(u \sum_i \lambda_i \right)}$,
where $u$ is an arbitrary real (see \Cref{sec:Harris}).
The satisfaction of this condition implies that the stationary measure $p$
is exponentially integrable \cite{MeynTweedie:1993}: for all $u>0$, we have 
\begin{eqnarray}\label{eq:expInt}
\Exp{V_u({\bm \lambda})} = \int_{\mathbb{R}^K}  e^{u\sum_i \lambda_i} \, p(d\bm{\lambda}) < \infty \, .
\end{eqnarray}
Exponential integrability implies the finiteness of the stationary moments of all orders.
Thus, within the context of finite LGL networks, the assumptions of bounded intensities
function $\phi_i$ is not required for the existence of stationary moments.

\begin{remark}
The regularity of the stationary measure of Galves-Locherbach networks has been studied in 
\cite{Locherbach:2018} under assumption of bounded intensity functions $\phi_i$ in $C^\infty(\mathbb{R})$.
In particular, a criterion is given for the stationary measure to admit a $C^k(\mathbb{R})$
density with respect to the Lebesgue measure on $\mathbb{R}$ for finite relaxation times $\inf_i \tau_i >0$.
\end{remark}


\subsection{Functional equation for generating functions}\label{sec:FunctionalAnalyis}

Within the stationary framework, it is natural to investigate the relation between low-dimensional
features of the dynamics, such as the moments of the invariant measure, and the structure of the network.
In particular, it would be highly desirable to express the individual mean spiking rates, i.e.,
the average intensities $\beta_i = \Exp{N_i((0,1])}$, in terms of the model parameters, namely
the time constants $\tau_i$, the base rates $b_i$, the reset values $r_i$, and most importantly, the synaptic weights $\mu_{ij}$.
However, direct analysis of the model via its infinitesimal generator does not provide any 
tractable characterization of the stationary moments $\beta_{n_1, \ldots, n_K}= \Exp{\lambda_1^{n_1} \ldots \lambda_1^{n_K}}$.
In fact, deriving equations for the moments $\beta_{n_1, \ldots, n_K}$ from the infinitesimal generator
would yield a non-closed hierarchy of equations, whereby equations characterizing moments of a given order
requires knowledge of moments of higher order \cite{Buice:2009aa,Ocker:2017aa}.

An alternative to such direct approaches consists in looking for equations satisfied by
functional transforms of $p$, such as the Laplace transform.
The reason for considering functional transforms is that at stationarity, one can exploit the
RCP \cite{BremaudBaccelli:2003} to exhibit a functional characterization of these transforms,
which can be solved by analytical methods for judiciously chosen functional transforms.
In practice, we find that the Laplace transform---or rather the moment-generating function (MGF)---of $p$
proves the most amenable for the analytical treatment of LGL networks.
By exponential integrability of the stationary distribution $p$ \eqref{eq:expInt}, the MGF of $p$ 
\begin{eqnarray}
\bm{u} = \lbrace u_1, \ldots , u_K \rbrace \mapsto L (\bm{u}) = \Exp{\exp{\left( \sum_{i=1}^K u_i \lambda_i \right)}} \, ,
\end{eqnarray}
is well-defined on all $\mathbb{R}_+^K$, and thus characterizes the probability distribution $p$.
In particular, the moments of $p$ can be derived from $L$ as
\begin{eqnarray}
m_{n_1, \ldots, n_K}= \Exp{\lambda_1^{n_1} \ldots \lambda_1^{n_K}}
=
 \frac{\partial^{\sum_i n_i} L}{ \prod_i \partial \lambda_i^{n_i}} \bigg \vert_{\bm{\lambda} = \bm{0}} \, .
\end{eqnarray}
The  MGF of the stationary distribution $p$ constitutes the functional transform of choice
for the analysis of LGL networks because it admits a simple characterization via RCPs:


\begin{proposition}\label{th:exactPDE}
The full $K$-dimensional MGF $L$ satisfies the first-order linear PDE
\begin{eqnarray}\label{eq:exactPDE}
\\
\left( \sum_i \frac{u_i b_i}{\tau_i}  \right) L - \sum_i \left( 1+\frac{u_i}{\tau_i} \right) \partial_{u_i} L
+
\sum_i e^{\left(u_i r_i + \sum_{j \neq i} u_j \mu_{ji}\right)} \partial_{u_i} L \Big \vert_{u_i=0} = 0 \, .
\nonumber
\end{eqnarray}
\end{proposition}


\begin{proof}
Given a function $f$ in the domain $\mathcal{D}(A)$, the $\mathcal{F}_t$-predictable process defined by
\begin{eqnarray}
f({\bm \lambda}(t)) -\int_0^t \mathcal{A}[f]({\bm \lambda}(s)) \, ds
\end{eqnarray}
is a martingale.
By stationarity of $\{ \lambda(t) \}_{t \in \mathbb{R}}$,
we have $\Exp{f({\bm \lambda}(t))} = \Exp{f({\bm \lambda}(0))}$ and Dynkin's formula reads
\begin{eqnarray}
\Exp{\int_0^t \mathcal{A}[f]({\bm \lambda}(s)) \, ds}
=
\int_0^t \Exp{\mathcal{A}[f]({\bm \lambda}(s))} \, ds
= 0 \, .
\end{eqnarray}
Moreover, also by stationarity of $\{ \lambda(t) \}_{t \in \mathbb{R}}$,
the expectation in the integrand is constant, i.e., 
$\Exp{\mathcal{A}[f]({\bm \lambda}(s))} = \Exp{\mathcal{A}[f]({\bm \lambda})}$ with:
\begin{eqnarray}
\Exp{\mathcal{A}[f]({\bm \lambda})} 
=
\sum_i  \Exp{\frac{b_i - \lambda_i}{\tau_i} \, \partial_{\lambda_i} f(\bm{\lambda})
+  \big( f(\bm{\lambda}+\bm{\mu}_i(\bm{\lambda})) - f(\bm{\lambda}) \big) \lambda_i }
=
0 \, .
\end{eqnarray}
Specializing the above relation to exponential functions $f({\bm \lambda}) = e^{\sum_i u_i \lambda_i}$ yields
\begin{eqnarray}
\sum_{i}  \Exp{ \frac{b_i - \lambda_i}{\tau_i} \, u_i e^{\sum_j u_j \lambda_j}+  \left( e^{u_i r_i + \sum_{j \neq i} u_j (\lambda_j + \mu_{ji}) } -  e^{\sum_{j} u_j \lambda_j} \right) \lambda_i }
= 0 \, .
\end{eqnarray}
which can be written under the form
\begin{eqnarray}\label{eq:fullPDE}
\lefteqn{
\sum_{i}  
\frac{b_i u_i}{\tau_i} \Exp{ e^{\sum_j u_j \lambda_j}}
- \sum_i \left( 1+\frac{u_i}{\tau_i} \right) \Exp{ \lambda_i e^{\sum_j u_j \lambda_j}} } \nonumber\\
&& \hspace{120pt}+
\sum_i e^{\left(u_i r_i + \sum_{j \neq i} u_j \mu_{ji}\right)} \Exp{ \lambda_i e^{\sum_{j \neq i} u_j \lambda_j}}
= 0 \, .
\end{eqnarray}
Equation \eqref{eq:exactPDE} follows from recognizing the expectation terms as values
of the MGF $L$ and its partial derivatives $\partial_{\lambda_i} L$.
\end{proof}

Equation \eqref{eq:fullPDE} is a non-local first-order linear partial differential equation (PDE)
with boundary terms involving partial derivatives. Conceptually, this equation can be viewed
as depicting the stationary state of a $K$-dimensional transport equation in the negative orthant,
with linear drift $( 1+u_1/\tau_1, \ldots, 1+u_K/\tau_K)$, with linear death rate $\sum_i b_i u_i/\tau_i$,
and with non-local birth rate related to fluxes through the hyperplane $\{ \lambda_i = 0 \}$, $1 \leq i \leq K$.
Despite this conceptual simplicity, the presence of flux-related, non-local,
birth rate precludes one from solving \eqref{eq:fullPDE} except for the simplest cases, i.e., for $K \leq 2$.
To gain knowledge about the typical state of LGL networks in the stationary limit,
one has to resort to approximation schemes, such as moment-truncation methods,
which can yield unphysical solutions without probabilistic interpretations and
are often analytically intractable \cite{Doiron:2016aa}.
The purpose of the present work is to introduce a computational framework circumventing
the above difficulties by studying replica versions of the LGL networks of interest,
which admit stationary states that are both probabilistically well-posed and analytically tractable.


\section{The Replica-mean-field approach}\label{sec:pproach}

In this section, we propose to decipher the activity of LGL networks via 
limit networks made of infinitely many replicas with the same basic network structure.
In \Cref{sec:RMF}, we define the RMF limit for LGL networks and the associated RMF \emph{ansatz},
a system of ODEs characterizing their stationary regime.  
In \Cref{sec:Palm}, we show that in practice, the RMF \emph{ansatz} can be derived without explicit
reference to the replica framework via a computational tool, called Palm calculus. 
In \Cref{sec:Sol}, we reduce the RMF \emph{ansatz} to a set of self-consistency equations 
specifying the stationary neuronal stochastic intensities.


\subsection{Replica-mean-field models}\label{sec:RMF}

Replica models are first rigorously defined for a finite number of replica and admit similar,
albeit higher dimensional, functional characterization as plain LGL networks.
However, in the RMF limit, the Poisson Hypothesis allows one to truncate correlation terms
due to neuronal interaction, yielding a set of ODEs characterizing the RMF stationary state.


\subsubsection{Finite-replica models} 
\label{sec:firemo}

In order-one replica models, each replica consists of the same number of neurons as the original LGL networks,
denoted by $K$, and within each replica, neurons are labelled by a class index $1 \le i \le K$.
For a finite model with $M$ replicas, let $N_{m,i}$ denote the point process representing the
spiking activity of the neuron of class $i$ in replica $m$, referred to as neuron $(m,i)$.
Moreover, let $\lbrace\lambda_{m,i} \rbrace_{1\le m\le M, 1\le i \le K}$, denote the corresponding stochastic intensity.
Instead of interacting with neurons in the same replica upon spiking, neuron $(m,i)$ interacts with target neurons
of classes $j \neq i$ from independently and uniformly chosen replicas and with synaptic weight $\mu_{ij}$.
Thus, replica models consist in a caricature of the initial model where the interactions between neurons
are randomized while keeping the finite structure of the original network.
The finite replica dynamics can be specified via the introduction of stochastic processes
registering the sequence of neuronal interactions across replicas.
For all $1\le m\le M, 1\le i \le K$, let $\lbrace v_{m,ij}(t) \rbrace_{t \in \mathbb{R}}$
be stochastic processes such that for every spiking time $T$, i.e., for every point of $N_{m,i}$, the random variables
$\lbrace v_{m,ij}(T) \rbrace_j$ are independent of the past, mutually independent,
and uniformly distributed over $\lbrace 1,\ldots,M \rbrace \setminus \lbrace m \rbrace$. 
Concretely,  $v_{m,ij}$ indicates the index of the replica containing the neuron of class $i$ targeted by neuron $(m,j)$ upon spiking.
Then, the stochastic intensities $\{ \lambda_{m,i} \}_{1\le m\le M, 1\le i \le K}$
characterizing the $M$-replica dynamics of the finite LGL network obey the following system of coupled stochastic equations:
\begin{eqnarray}\label{eq:dynmod-rep}
\lambda_{m,i}(t) & = & \lambda_{m,i}(0) +
\frac{1}{\tau_i} \int_0^t \big(b_i-\lambda_{m,i}(s) \big)\, ds\nonumber \\
&+&  
\sum_{n \neq m} \sum_{j \neq i} \mu_{ij} \int_0^t  1_{\lbrace v_{n,ij}(s)=m \rbrace} N_{n,j}(ds) + \int_0^t \big(r_i-\lambda_{m,i}(s) \big) N_{m,i}(ds) \, .
\end{eqnarray}
These equations, which generalize \eqref{eq:dynmod}, 
entirely define the Markovian dynamics of finite replica models for LGL networks.
Similarly, the infinitesimal generator \eqref{eq:infGen} can be generalized to the finite replica setting.
To account for randomized interactions, let us introduce the $K$-dimensional stationary random vectors $\bm{v}_{m,i}$, defined by
$\left[ \bm{v}_{m,i} \right]_j = v_{m,ij}(T_{m,i,0})$ if $j \neq i$ and $\left[ \bm{v}_{m,i} \right]_i = m$,
taking values in the set of integers
\begin{eqnarray}
V_{m,i} = \Big \lbrace \bm{v} \in [1 \ldots M]^K \, \vert \, v_i= m \quad \mathrm{and} \quad v_j \neq m \, , j \neq i \Big \rbrace \, ,
\end{eqnarray}
whose cardinality is $\vert V_{m,i} \vert = (M-1)^{K-1}$. 
By definition, the collection of vectors $\bm{v}_{m,i}$, which indicates the target neurons of neuron $(m,i)$,
are identically and uniformly distributed on the sets $V_{m,i}$.
Consequently, the infinitesimal generator for the $M$-replica Markovian dynamics can be written as
\begin{eqnarray}\label{eq:infGen2}
\mathcal{A}[f_{\bm{u}}](\bm{\lambda}) &=& \sum_{i=1}^{K} \sum_{m=1}^{M} \left( \frac{b_i - \lambda_{m,i}}{\tau_i} \right) \partial_{\lambda_{m,i}} f_{\bm{u}}(\bm{\lambda})  \nonumber\\
&+& \sum_{i=1}^{K} \sum_{m=1}^{M}  \frac{1}{\vert V_{m,i} \vert} \sum_{\bm{v} \in V_{m,i}} \Big( f(\bm{\lambda} + \bm{\mu}_{m,i,\bm{v}}(\bm{\lambda})) -  f(\bm{\lambda} ) \Big) \lambda_{m,i} \, ,
\end{eqnarray}
where the update due to the spiking of neuron $(m,i)$ is defined by
\begin{eqnarray}
\Big[ \bm{\mu}_{m,i,\bm{v}}(\bm{\lambda}) \Big]_{j,n} = 
\left\{
\begin{array}{ccc}
 \mu_{ji} & \mathrm{if}  &  j \neq i \, , \: n=v_j \, , \\
r_i-\lambda_{m,i}  & \mathrm{if}   & j = i \, , \: n=v_j \, , \\
0  & \mathrm{otherwise  .}  &   
\end{array}
\right.
\end{eqnarray}
The arguments developed in \Cref{sec:MarkovAnalyis} for the Markovian analysis of plain LGL networks
naturally extend to finite replica models.
In particular, $M$-replica networks are Harris ergodic and admit a stationary distribution $p$. 
In turn, we can apply the RCP of \Cref{sec:FunctionalAnalyis} to the
stationary $M$-replica dynamics to obtain a functional characterization for the MGF of $p$:
\begin{eqnarray}
\bm{u} \mapsto L (\bm{u}) = \Exp{\exp{\left( \sum_{m=1}^M \sum_{i=1}^K u_i \lambda_{m,i} \right)}} \, .
\end{eqnarray}
Specifically, in \Cref{sec:funcRepModel},
we show the following result on the LGL networks defined in \Cref{sec:LGL}:


\begin{proposition}\label{th:exactPDErep}
For all LGL networks, the $M$-replica MGF $L$ satisfies the first-order linear PDE
\begin{eqnarray}\label{eq:fullPDErep}
\lefteqn{
\sum_m \sum_i\frac{b_i u_{m,i}}{\tau_i} L({\bm u})
- \sum_m \sum_i \left( 1+\frac{u_i}{\tau_i} \right) \partial_{\lambda_{m,i}} L({\bm u})  } \nonumber\\
&& \hspace{80pt}+
\sum_m \sum_i \frac{1}{\vert V_{m,i} \vert} \sum_{\bm{v} \in V_{m,i}} e^{\left(u_{m,i} r_i + \sum_{j \neq i} u_{v_j\!,j} \mu_{ji}\right)} L({\bm u})
= 0 \, .
\end{eqnarray}
\end{proposition}
The above characterization of replica networks is not simpler than that of plain LGL networks.
However, the expression of the infinitesimal generator \eqref{eq:infGen2} shows that randomized interactions effectively implement an averaging over replicas.
In the limit of a large number of replicas $M \to \infty$, one expects such an averaging to erase
the dependence structure of spiking interactions, and to yield independence between replicas.
Numerical simulations support such a mean-field behavior, which is conceptually similar to that of the
thermodynamic limit, i.e., with $K \to \infty$ and vanishing interactions scaling as $1/K$,
but retains important features of the finite network structure.
Intuitively, independence between two replicas emerges from the so-called
``Poisson Hypothesis'' \cite{RybShlosI,RybShlosII}:
Over a finite period of time, the probability for a particular neuron to receive
a spike from another given neuron scales as $1/M$.
Thus, as the number of replicas increases, interactions between distinct replicas become ever scarcer, leading to replica independence.  
By the same intuition, we expect spiking deliveries to distinct replicas to be asymptotically
distributed as independent Poisson point processes, which is precisely the Poisson Hypothesis.
Proving the validity of the Poisson Hypothesis requires to establish the property of propagation of chaos
\cite{Sznitman:1989} in the limit of an infinite number of replicas $M \to \infty$. 
This is beyond the aims of our analyis.
Here, we conjecture that the Poisson Hypothesis holds in the limit $M \to \infty$,
and our goal is to develop the computational framework for the analysis of infinite-replica LGL networks,
which we refer to as RMF models.


\subsubsection{The replica-mean-field \emph{ansatz}}\label{sec:RMFA}
Under the Poisson Hypothesis, neurons from distinct replicas of an RMF network spike independently.
Here, we show that this assumption of independence leads to a simple functional characterization
of the MGF of a single replica, which we call the RMF \emph{ansatz}.
Consider for instance the MGF associated to the first replica:
\begin{eqnarray}
\bm{u} \mapsto L (\bm{u}) = \Exp{\exp{\left(\sum_{i=1}^K u_{i,1} \lambda_{1,i} \right)}} \, .
\end{eqnarray}
Denoting $u_{i,1}=u_1$ and  $ \lambda_i = \lambda_{1,i}$ for conciseness, 
the RCP for the $M$-replica network applied
to $f(\bm{u}) = e^{\sum_{i=1}^K u_i \lambda_i}$ (see \Cref{sec:funcRepModel}) yields
\begin{eqnarray}\label{eq:fullPDErep2}
\lefteqn{
 \sum_{i=1}^{K}  \left( \frac{b_i u_{i}}{\tau_i} L(\bm{u}) -  \frac{u_{i}}{\tau_i}  \partial_{u_{i}}L(\bm{u}) \right) + \sum_{i=1}^{K}   \left( e^{u_{i}r_i}- 1 \right)  \partial_{u_{i}}L(\bm{u}) \big \vert_{u_i=0} } \nonumber\\
&& \hspace{40pt}+ \sum_{i=1}^{K} \sum_{m >1}  \frac{1}{\vert V_{m,i} \vert} \sum_{\bm{v} \in V_{m,i}}\left( e^{\left(\sum_{j \neq i, v_j =1} u_{j} \mu_{ji} \right)} - 1 \right) \Exp{\lambda_{m,i} e^{\sum_{i=1}^K u_i \lambda_i}} = 0\, .
\end{eqnarray}
The above equation would constitute an autonomous ODE for $L(u)$,
were it not for the interactions with replicas $M>1$, as mediated by the last term of  \eqref{eq:fullPDErep2}.
The independence assumption of the Poisson Hypothesis allows us to close \eqref{eq:fullPDErep2}
in the limit of an infinite number of replica $M \to \infty$.
The first step in this direction is to observe that in the limit $M \to \infty$, 
only certain vectors $\bm{v}$ contribute meaningfully to the interaction terms:
these are those vectors representing spike deliveries from a neuron $(m,j)$, $m>1$,
such that only one spike is delivered to the first replica.
In fact, we elaborate on this observation in \Cref{sec:funcRepModel} to show that 
\begin{eqnarray}\label{eq:prepAnsatz}
\lefteqn{
\sum_{i=1}^{K} \sum_{m >1}  \frac{1}{\vert V_{m,i} \vert} \sum_{\bm{v} \in V_{m,i}}\left( e^{\left(\sum_{j \neq i, v_j =1}u_{j} \mu_{ji} \right)} - 1 \right) \Exp{\lambda_{m,i} e^{\sum_{i=1}^K u_i \lambda_i}}  = }\nonumber\\
&& \hspace{60pt}\sum_{i=1}^{K} \sum_{j \neq i}  \left( e^{u_{j} \mu_{ji}} -1 \right) \frac{1}{M-1} \sum_{m >1}\Exp{\lambda_{m,i} e^{\sum_{i=1}^K u_i \lambda_i}} + o(1/M)\, .
\end{eqnarray}
By exchangeability of the replicas, all expectation terms in the right-hand side above are equal.
Moreover, neurons of the same class have identical mean intensities: $\beta_i = \Exp{\lambda_{m,i}}$.
Exploiting the assumption of independence from the Poisson Hypothesis, we thus have
\begin{eqnarray}
\Exp{\lambda_{m,i} e^{\sum_{i=1}^K u_i \lambda_i}} = \Exp{\lambda_{m,i}}  \Exp{e^{\sum_{i=1}^K u_i \lambda_i}} = \beta_i L(\bm{u}) \, .
\end{eqnarray}
Using the fact that we also have $\beta_i = \partial_{u_i} L (\bm{u})  \vert_{u_i=0}$,
we can write  \eqref{eq:fullPDErep2} as
\begin{eqnarray}\label{eq:fullPDErep3}
\lefteqn{  \sum_{i=1}^K -\frac{u_i}{\tau_i}  \partial_{u_i}L(\bm{u}) +  \sum_{i=1}^K \left( \frac{u_i b_i}{\tau_i} + \sum_{j \neq i}\left( e^{u_i \mu_{ij}}-1\right) \beta_j \right) L(\bm{u}) \: +} \nonumber\\
&& \hspace{200pt}   \left(e^{u_i r_i} -1 \right) \partial_{u_{i}}L(\bm{u}) \big \vert_{u_i=0}= 0 \, .
\end{eqnarray}
The above equation is separable. In keeping with the assumption of independence,
plugging in the product form $L(\bm{u}) = \prod_{i=1}^K L_i(u_i)$ with $L_i(u_i) = \Exp{e^{u_i \lambda_i}}$ 
and $\beta_i = L_i'(0)$, yields the final form of the RMF \emph{ansatz}:


\begin{definition}\label{def:relaxingAnsatz}
The RMF \emph{ansatz} for the LGL network of $K$ neurons specified by the interaction weights
$\mu_{ij}$, the relaxation times $\tau_i$, the base rates $b_i$, and by the reset values $r_i$, $1 \leq i \leq K$,
is defined as the system of coupled ODEs:
\begin{eqnarray}\label{eq:sysLaplace}
-\left( 1+ \frac{u}{\tau_i}\right)  L_i'(u) +  \left( \frac{ub_i}{\tau_i}  +\sum_{j \neq i}\left( e^{u \mu_{ij}}-1\right) \beta_j \right) L_i(u) +  \beta_i e^{u r_i} = 0 \, .
\end{eqnarray}
\end{definition}


Notice that setting $u \to 0$ in \eqref{eq:sysLaplace} automatically yields $L_i'(0)= \beta_i$.
Thus, at the cost of introducing the mean firing rates $\bm{\beta} = \lbrace \beta_1, \ldots, \beta_K\rbrace$,
the Poisson Hypothesis allows us to write a closed set of ODEs for the 
one-dimensional MGF $L_i$, should the RMF \emph{ansatz} be true.
However, in the RMF \emph{ansatz}, the mean firing rates $\bm{\beta}$ are 
unknown parameters, and the MGF normalization condition, $L_i(0)=1$, does not dispel this indetermination.
More generally, there is \emph{a priori} no reason for the RMF \emph{ansatz} to admit a MGF as a solution.
In the following, we show that for the RMF \emph{ansatz} to admit a MGF solution, 
$\bm{\beta}$ needs to solve a set of self-consistency equations.

We will first account for this result in the special case of the 
counting-neuron model, i.e., for a fully connected network with
homogeneous synaptic weights and without relaxation:  $\mu_{ij} = \mu$ and $\tau_i \to \infty$.
For the counting-neuron model, it is best to work with the probability-generating function (PGF)
associated to the counting vector $\bm{C}=\lbrace C_i, \ldots, C_n\rbrace$:
\begin{eqnarray}
\bm{z} \in [0,1]^K \mapsto G (\bm{z}) = \Exp{\prod_{i=1}^K  z_i^{C_i}} = L (\ln z_{i_1}, \ldots, \ln z_{i_K})\, ,
\end{eqnarray}
rather than with the actual MGF of $\bm{C}$, still denoted by $L$.
Specifically, we have:


\begin{definition}\label{def:countingAnsatz}
The RMF \emph{ansatz} for the network of $K$ node counting neuron network
specified by the interaction weight $\mu$, and the reset values $r$, $1 \leq i \leq K$,
is defined as the ODE:
\begin{eqnarray}\label{eq:firstOrder}
\beta - \mu z G'(z) + \big( \beta (K-1)(z-1) -r\big) G(z) = 0 \, .
\end{eqnarray} 
\end{definition}


Before proceeding to the reduction of the RMF \emph{ansatz} to a set of self-consistency 
equations for $\bm{\beta}$, we show that the RMF \emph{ansatz} can be obtained without
any explicit reference to replica models.
In doing so, our aim is to show that the RMF \emph{ansatz} can be established
intuitively via independence assumptions, and without in-depth probabilistic analysis.


\subsection{Functional equations via Palm calculus}\label{sec:Palm}

The derivation of the RMF \emph{ansatz} relies on a computational tool
from the theory of point processes, called Palm calculus \cite{Mathes:1964,Mecke:1967}.

\subsubsection{Primer on Palm calculus}

Palm calculus treats stationary point processes from the point of view of a typical point, i.e.,
a typical spike, rather than from the point of view of a typical time, i.e., in between spikes.
Here, we only introduce Palm calculus via the two formulae that play a key role in deriving
the RMF {\it ansatz} \cite{BremaudBaccelli:2003}.
With no loss of generality, consider a stationary point process $N_i$ defined on some probability space
$(\Omega, \mathcal{F},\mathbb P)$, representing the spiking activity of a neuron.
If $\{\theta_t\}$ is a time shift on $(\Omega,{\mathcal F})$ which preserves $\mathbb P$, 
we say that the stationary point process $N$ is $\theta_t$-compatible in the sense that $N(B)\circ \theta_t=N(B+t)$ for all 
$B$ in $\mathcal{B}(\mathbb{R})$ and $t\in \mathbb R$.
With this notation, the Palm probability of $N$, which gives the point of view of a ``typical'' point on $N$,
is defined on $(\Omega, \mathcal{F})$ for all event $A$ in $\mathcal{F}$ and for all time $t>0$ by
\begin{eqnarray}\label{eq:Palm}
\hspace{15pt}
{\mathbb P}^0_N(A) = \frac{1}{\beta t} \Exp{\sum_{n \in \mathbb{Z}} \mathbbm{1}_A(\theta_{T_{n}})  \mathbbm{1}_{(0,t]}(T_{n})}
= \frac{1}{\beta t} \Exp{\int_{(0,t]} \left( 1_A \circ \theta_s \right) N(ds)} \, ,
\end{eqnarray}
where $\beta = \Exp{N((0,1])}$.
Informally, ${\mathbb P}^0_N(A)$ represents the conditional probability
that a train of spikes falls into $A$ knowing that a spike happens at $t=0$.
Moreover, suppose that $N$ admits a stochastic intensity $\lambda_i$,
representing the instantaneous firing rate, and set $A= \{ \lambda(0) \in B\}$ for some $B$ in $\mathcal{B}(\mathbb{R}_+)$, then
\begin{eqnarray}
{\mathbb P}^0_N(A)  = {\mathbb P}^0_N \left[ \lambda (0_-) \in B  \right] = \mathbb P \left[ \lambda(0_-) \in B \, \vert \, N(\{0\})=1 \right] \, 
\end{eqnarray}
specifies the stationary law of the stochastic intensity $\lambda_i$ just before spiking.

The notions of Palm probability and stochastic intensity provide the basis for the theory of Palm calculus.
Let us consider another non-negative stochastic process $X$ defined on the same underlying probability 
space $(\Omega, \mathcal{F})$ as that of $N$.
If $X$ is also $\theta_t$-compatible in the sense that $X(s)\circ \theta_t=X(s+t)$ for all $t,s\in \mathbb R$,
then the first key formula Palm calculus directly follows from the definition \eqref{eq:Palm} and reads
\begin{eqnarray}\label{eq:Palm1}
\ExpPN{X(0_-)}
=
\frac{1}{\beta t} \Exp{\int_0^t X (s) N(ds)} \, ,
\end{eqnarray}
where $\ExpPN{ \cdot}$ denotes the expectation with respect to ${\mathbb P}^0_N$.
In the following, the process $X$ intervening in the above expression will typically
be a function of the stochastic intensity of a neuron.
The second key formula, which follows from the Papangelou theorem,
relates Palm probabilities to the underlying probability via the notion of stochastic intensity \cite{BremaudBaccelli:2003}.
Specifically, if $N$ admits a stochastic intensity $\lambda$ and $X$ has appropriate
predictability properties, then for all real valued functions $f$ we have:
\begin{eqnarray}\label{eq:Palm2}
\Exp{f(X) \lambda_i}
=
\beta \ExpPN{f\big(X(0_-)\big)}  \, .
\end{eqnarray}
The formulae \eqref{eq:Palm1} and \eqref{eq:Palm2} will be the only results required to
establish rate-conservation equations via Palm calculus.


\subsubsection{Rate-conservation equations}

Because interactions are temporally localized at spiking times,
Palm calculus is a convenient tool to express rate-conservation equations in LGN networks.
In fact, Palm calculus allows one to obtain rate-con\-servation equations intuitively from the
stochastic equations describing the evolution of the conserved quantity. 
For our purpose of recovering the RMF \emph{ansatz} from \Cref{def:relaxingAnsatz},
that conserved quantity is $e^{u \lambda_i}$, where $u$ is some fixed real and where
$\lambda_i$ is the stochastic intensity of neuron $i$, $0 \leq i \leq K$.
By $\mathcal{F}_t$-predictability and stationarity of the network dynamics $\bm{\lambda}_t$,
for all real $u$, the process $\lbrace e^{u \lambda_i(t)} \rbrace_{t \in \mathbb{R}}$
is also a $\mathcal{F}_t$-predictable stationary process. Moreover, this process satisfies the stochastic equation
\begin{eqnarray}\label{eq:relaxStoch}
\lefteqn{
e^{u \lambda_i(t)} = e^{u \lambda_i(0)} + \frac{u}{\tau_i} \int_0^t \big( b_i - \lambda_i(s)\big) e^{u \lambda_i(s)} \, ds } \nonumber\\
&&\hspace{40pt}+ \sum_{j \neq i} \left( e^{u \mu_{ij}}-1\right) \int_0^t e^{u \lambda_i(s)} N_j(ds) + \int_0^t \left(   e^{u r_i} - e^{u \lambda_i(s)} \right) N_i(ds) \, ,
\end{eqnarray}
where the $N_i$, $0 \leq i \leq K$, are $\mathcal{F}_t-$predictable counting processes with stochastic intensity $\lambda_i$.
In  \eqref{eq:relaxStoch}, the first integral term is due to relaxation toward base rate $b_i$,
the second integral term is due to interaction with spiking neurons $j \neq i$,
and the last term is due to post-spiking regeneration of neuron $i$ at reset value $r_i$.
Taking the expectation of \eqref{eq:relaxStoch} with respect to the stationary
measure of $\bm{\lambda}$ yields the rate-conservation equations of $\lbrace e^{u\lambda_i(t)} \rbrace_{t \in \mathbb{R}}$:
\begin{eqnarray}
\lefteqn{
\frac{u}{\tau_i}\Exp{\int_0^t \big( b_i - \lambda_i(s)\big) e^{u \lambda_i(s)} \, ds }  } \nonumber\\
&&  \hspace{30pt} + \sum_{j \neq i} \left( e^{u \mu_{ij}}-1\right) \Exp{\int_0^t e^{u \lambda_i(s)} N_j(ds) } +  \Exp{ \int_0^t \left(   e^{u r_i} - e^{u \lambda_i(s)} \right)  N_i(ds)} = 0 \, , 
\end{eqnarray}
where we have used that by stationarity, we have $\Exp{e^{u \lambda_i(t)}} = \Exp{e^{u \lambda_i(0)}}=\Exp{e^{u \lambda_i}}$.
Again, by stationarity, the expectation of the relaxation integral term can be expressed as 
\begin{eqnarray}
\Exp{\int_0^t \big( b_i - \lambda_i(s)\big) e^{u \lambda_i(s)} \, ds } =  t \Exp{ (b_i-\lambda_i) e^{u \lambda_i}} \, ,
\end{eqnarray}
where $\beta_i = \Exp{\lambda_i}= \Exp{N_i((0,1])}$ is the mean intensity of $N_i$.
In turn, introducing the Palm distribution ${\mathbb P}^0_i$ of the process $\bm{\lambda}$
with respect to $N_i$ allows us to write the expectations of the remaining interaction and
reset integral terms as expectations with respect to Palm distributions ${\mathbb P}^0_i$, $1 \leq i \leq K$.
Specifically, by applying formula \eqref{eq:Palm1}, we have
\begin{eqnarray}\label{eq:rateconv2}
\Exp{\int_0^t e^{u \lambda_i(s)} N_j(ds) } &=& \left(\beta_j t \right) \Expj{ e^{u \lambda_i(0^-)}} \, , \\
\Exp{ \int_0^t \left(   e^{u r_i} - e^{u \lambda_i(s)} \right) N_i(ds)} &=&  \left( \beta_i t \right) \Expi{  e^{u r_i} - e^{u \lambda_i(0^-)} } \, ,
\end{eqnarray}
where $\Expi{\cdot}$ denotes expectation with respect to ${\mathbb P}^0_i$.
With these observations, the rate-conservation equation can be expressed under a local form, i.e., 
without integral terms, but at the cost of taking expectation with respect to distinct probabilities:
\begin{eqnarray}\label{eq:rateConsPalm}
\lefteqn{
\frac{u}{\tau_i}  \Exp{ (b_i-\lambda_i) e^{u \lambda_i}} } \nonumber\\
&& \hspace{40pt} + \sum_{j \neq i} \left( e^{u \mu_{ij}}-1\right) \beta_j \Expj{e^{u \lambda_i(0^-)}}
+  \beta_i  \Expi{  e^{u r_i} - e^{u \lambda_i(0^-)} } = 0 \, . 
\end{eqnarray}
The above equation can then be expressed under a local form involving only
the stationary measure thanks to Papangelou's theorem \eqref{eq:Palm2}, allowing us to write
\begin{eqnarray}
\hspace{30pt} \beta_j \Expj{  e^{u \lambda_i(0_-)}} = \Exp{ \lambda_j  e^{u \lambda_i}} \quad \mathrm{and} \quad \beta_i \Expi{  e^{u \lambda_i(0_-)}} = \Exp{ \lambda_i  e^{u \lambda_i}} \, .
\end{eqnarray}
Using the above relations in \eqref{eq:rateConsPalm}, the final form of the
exact rate-conservation equations of $\lbrace e^{u\lambda_i(t)} \rbrace_{t \in \mathbb{R}}$, $1 \leq i \leq K$, is
\begin{eqnarray}\label{eq:relaxCons}
\lefteqn{
-\left( 1+ \frac{u}{\tau_i}\right)  \Exp{ \lambda_i  e^{u \lambda_i}}  + \frac{ub_i}{\tau_i} \Exp{e^{u \lambda_i}} } \nonumber\\
&& \hspace{80pt} + \sum_{j \neq i}\left( e^{u \mu_{ij}}-1\right) \Exp{ \lambda_j  e^{u \lambda_i}} +  \beta_i e^{u r_i} = 0 \, ,
\end{eqnarray}
where we have dropped time dependence for stationary random variables.


\subsubsection{Moment truncation}
Applying the RCP under the Poisson Hypothesis effectively truncates
correlation terms due to interactions in the exact rate-conservation equation of replica models.
Although not apparent in the Markovian treatment of \Cref{sec:RMFA}, such a truncation become
straightforward when working on the rate-conservation equation \eqref{eq:relaxCons} obtained via Palm calculus.
Indeed, \eqref{eq:relaxCons} can be interpreted as a differential equation for the one-dimensional MGF
of $\bm{\lambda}$ defined by $L_i(u) = \Exp{e^{u \lambda_i}}$ for all $i$.
However,  \eqref{eq:relaxCons} for $L_i$ involves the second-order statistics
of $\bm{\lambda}$ via the terms $\Exp{ \lambda_j  e^{u \lambda_i}}$,
which is not captured by $L_i$ but by the two-dimensional MGFs of $\bm{\lambda}$.
Not surprisingly, making the Poisson Hypothesis allows one to close  \eqref{eq:relaxCons},
as it implies that the stochastic intensities of distinct neurons are independent variables:
\begin{eqnarray}
\Exp{ \lambda_j e^{u \lambda_i}}   = \beta_j \Exp{e^{u \lambda_i}} \quad \mathrm{for} \quad j \neq i \, .
\label{eq:independentization}
\end{eqnarray}
Thus, under the Poisson Hypothesis, \eqref{eq:relaxCons} becomes an equation about the random variable $\lambda_i$ alone:
\begin{eqnarray}
\lefteqn{
-\left( 1+ \frac{u}{\tau_i}\right)  \Exp{ \lambda_i  e^{u \lambda_i}}   } \nonumber\\
&& \hspace{50pt} +  \left( \frac{ub_i}{\tau_i}+\sum_{j \neq i}\left( e^{u \mu_{ij}}-1\right) \beta_j \right) \Exp{e^{u \lambda_i}} +  \beta_i e^{u r_i} = 0 \, .
\label{eq:direct}
\end{eqnarray}

The above equation is precisely that intervening in the mean-field-replica \emph{ansatz} in \Cref{def:relaxingAnsatz}.
As announced, it has been obtained by truncation of the rate-conservation equations via
Palm calculus and without any explicit reference the RMF network.
Considering \eqref{eq:direct} as a heuristic simplification of \eqref{eq:relaxCons} leads to a natural question: why should the heuristic simplification  based on \eqref{eq:independentization}
lead to some equation having a probabilistic interpretation? The RMF framework provides the answer to this question:
the RMF network is a stochastic dynamical system whose steady-state MGF should satisfy \eqref{eq:direct}.
In other words, the existence of a steady state for the RMF network, which is conjectured here, justifies
the existence of at least one probabilistic solution to \eqref{eq:direct}.
As stated previously, proving rigorously the existence of that steady state consists in establishing the property of propagation of chaos
\cite{Sznitman:1989} in RMF networks, which is beyond the aims of our analysis.


\subsection{Analytical solutions for replica-mean-field models}\label{sec:Sol}

The rate-con\-servation equations appearing in the RMF \emph{ansatz} are first-order ODEs. 
Hence, characterizing the stationary state of RMF networks amounts to specifying the unknown mean intensities
featuring in these differential equations.
Intuitively, the mean intensities must solve a set of self-consistency equations:
for each neuron, $\beta_i$ is the output firing rate of a neuron subjected to input firing rates
$\beta_j$ delivered via synaptic weight $\mu_{ij}$.
The goal of this section is twofold: first, we derive such self-consistency equations 
via simple analyticity requirements of the solutions of the differential equations. 
Second, we numerically validate the properties of the RMF framework by comparison
with the original LGL network or with the classical thermodynamic limit.

\subsubsection{The counting model case}

The analytical strategy that we will follow for general LGL models is first exemplified on the simplest network, i.e.,
the counting model with $K$ fully connected neurons with homogeneous synaptic weights $\mu$ and with uniform base rate $b$.
By neuronal exchangeability, the RMF \emph{ansatz} for the counting model (see \Cref{def:countingAnsatz}) 
takes the form of a single equation for the PGF of $C$, the number of spikes received by a neuron since the last reset.
Then, for any $\beta$, that equation admits a unique solution $G$ satisfying the normalization condition that $G(1)=1$,
thereby defining a family of candidate PGFs $\lbrace G_\beta \rbrace_\beta$, parameterized by the unknown $\beta$.
As explained above, the RMF \emph{ansatz} should have at least one solution $G_\beta$ which is a PGF.
It turns out that, for the counting model, requiring the analyticity of the solutions in zero is enough
to determine a unique PGF solution to the RMF \emph{ansatz}. 
Specifically, we show in the following that, given the normalization condition $G(1)=1$,
there is a unique continuous solution to the RMF \emph{ansatz} and that the normalization condition
for that solution yields the self-consistency equation for $\beta$.
Moreover, we are able to show that this equation uniquely specifies $\beta$ and that
the corresponding function $G_\beta$ is indeed a PGF by explicitly exhibiting the associated
stationary probability distribution. These results are summarized in the following theorem:


\begin{theorem}\label{th:counting}
For the counting model,
there is a unique integer-valued random variable $C$ whose PGF is solution to the RMF \emph{ansatz}  \Cref{def:countingAnsatz}. 
Moreover, $(i)$ the mean intensity $\beta = b+\mu\Exp{C}$ is the unique solution to:
\begin{eqnarray}\label{eq:selfCons1}
\beta = \frac{\mu c^a e^{-c}}{\gamma (a,c)} \quad \mathrm{with} \quad a = \frac{(K-1)\beta+b}{\mu} \quad \mathrm{and} \quad c = \frac{(K-1)\beta}{\mu} \, ,
\end{eqnarray}
where $\gamma$ denotes the lower incomplete Gamma function, and $(ii)$ the stationary distribution of $C$ is given by 
\begin{eqnarray}
p(n) =
\left\{
\begin{array}{ccc}
\displaystyle \frac{\beta}{\mu a} = \frac{\beta}{(K-1)\beta+b} \, , & \mathrm{if} &  n=0 \, ,\\
\\
\displaystyle \frac{c^a e^{-c}}{\gamma(a,b)} \frac{\Gamma(a+n+1)c^n}{\Gamma(a)\Gamma(n+1)}  \, , & \mathrm{if} &  n>0 \, .
\end{array}
\right.\nonumber\\
\end{eqnarray}
\end{theorem}

\begin{proof}
The unique solution to the first-order differential equation \eqref{eq:firstOrder}
that satisfies the normalization condition $G(1)=1$ is
\begin{eqnarray}
G(z)=
\frac{e^{c(z-1)}}{ z^a} \left(1+ \frac{\beta e^{c}}{\mu c^a} \big(\Gamma(a,c)-\Gamma(a,cz)\big)  \right) \, ,
\end{eqnarray}
where $\Gamma$ denotes the upper incomplete Gamma function, i.e., $\Gamma(x,y) = \int_y^\infty t^{x-1}e^{-t} \, dt$,
and where we have used the auxiliary parameters $a$ and $c$  defined in \eqref{eq:selfCons1}.
Solutions $G$ are analytic on $\mathbb{R}$ except possibly in zero, where $G$ generically has an infinite discontinuity.
Indeed, noting that $a>0$, we have the following asymptotic behavior when $z \to 0^+$:
\begin{eqnarray}
G(z) =  z^{-a} \left(e^{-c}+\frac{\beta ( \Gamma (a,c)- \Gamma (a))}{\mu b^a}\right) + \frac{\beta}{\mu a}+O\left(z\right) \, .
\end{eqnarray}
As probability-generating functions must be analytic in zero, we require the term between parentheses
to be zero in the above expression, which is equivalent to requiring that $\beta$ solves the leftmost
equation of \eqref{eq:selfCons1}.
Observing that $(K-1)\beta = c/\mu$ and $a=c+b/\mu$,  \eqref{eq:selfCons1} can be rewritten as an equation on $c$:
\begin{eqnarray}
c^{1-\left(c+\frac{b}{\mu}\right)}e^c \, \gamma \left(c+\frac{b}{\mu} , c\right) = K-1
\end{eqnarray}
Then, applying Lemma \ref{lm:counting} (see below) with $x=b/\mu$ and $y=K-1$ shows that
Equation \eqref{eq:selfCons1} admits a unique solution for $b>0$, $\mu > 0$ and $K>0$.
The result for $\mu=0$, i.e., for independent neurons, is clear: $\lambda = b$.
For $\beta$ solving \eqref{eq:selfCons1}, the solution to \eqref{eq:firstOrder} can be written 
\begin{eqnarray}
G(z)=
\frac{e^{c(z-1)}}{ z^a} \frac{\gamma(a,zc)}{\gamma(a,c)}   \, ,
\end{eqnarray}
and repeated differentiations shows that $G$ is the PGF associated
to the distribution  defined over the integers by 
\begin{eqnarray}
p(n) = \frac{G^{(n)}(0)}{n!} = \frac{c^a e^{-c}}{\gamma(a,b)} \frac{c^n/n! }{a(a+1)\ldots(a+n)} = \frac{c^a e^{-c}}{\gamma(a,b)} \frac{\Gamma(a+n+1)c^n}{\Gamma(a)\Gamma(n+1)}  \, ,
\end{eqnarray}
and for which we have
\begin{eqnarray}
p(0) = \frac{\beta}{\mu a} = \frac{\beta}{(K-1)\beta+b} \leq 1 \, .
\end{eqnarray}
\end{proof}

The proof of \Cref{th:counting} utilizes the following lemma:


\begin{lemma}\label{lm:counting}
For all $x, y \geq 0$, there is a unique positive real $c$ such that
\begin{eqnarray}\label{eq:lmCounting}
c^{1-(x+c)}e^c \gamma(x + c, c) = y \, ,
\end{eqnarray}
where $\gamma$ denotes the lower incomplete Gamma function.
\end{lemma}


\begin{proof}
The power series representation of the incomplete Gamma function yields
\begin{eqnarray}
f(c) = c^{1-(x+c)}e^c \gamma(x + c, c) = \sum_{n=0}^\infty \frac{c^{n+1}}{(x+c)(x+c+1) \ldots (x+c+n)}\, ,
\end{eqnarray}
where the series converges uniformly in $c$ on all compacts in $\mathbb{R}_+$.
Denoting the continuous summand functions by
\begin{eqnarray}
f_n(c) = \frac{c^{n+1}}{(x+c)(x+c+1) \ldots (x+c+n)} \, ,
\end{eqnarray}
we observe that $f_n$ is differentiable on $\mathbb{R}_+^*$ with
\begin{eqnarray}
f'_n(c)=\frac{c^n \left(n+1 - \sum_{m=0}^{n} \frac{c}{x+c+m}\right)}{(x+c)(x+c+1) \ldots (x+c+n)} > 0 \, .
\end{eqnarray}
Thus, by uniform convergence, $f$ is a strictly increasing continuous function.
To prove the lemma, we need to show that $f$ is onto $\mathbb{R}_+$, i.e. that $\lim_{c \to \infty} f(c) = \infty$ since $f(0)=0$.
This limit directly follows from the positivity of $f_n$ on $\mathbb{R}_+$ and from the fact that $\lim_{c \to \infty} f_n(c) =1$ for all $n \geq 0$.
\end{proof}


\begin{remark}
The generating function $G$ obtained by solving for $\mu = 0$ and $\beta = b$
\begin{eqnarray}
G(z) = \frac{b}{b+(K-1)(1-z)} \, ,
\end{eqnarray}  
is the PGF of a geometric distribution with parameter $(1+(K-1)/b)^{-1}$, which is precisely the law of independent Poissonian arrivals during an exponential waiting time, i.e., the law of the spike count of a neuron during the inter-spike period of another. 
In particular, the mean count value is $G'(z) = (K-1)/b$, as expected.
\end{remark}


\begin{remark}
While neglecting coupling between neurons, the stationary distribution $p$ incorporates self-excitation via interaction-dependent mean intensities and also captures the effect of spiking reset.
For instance, keeping  $a-c=b/\mu$ and letting $a \to \infty$, as in the limit of large $K$, we have
\begin{eqnarray}
\frac{b^a e^{-b}}{\Gamma (a)-\Gamma (a,b)} = \sqrt{\frac{2 a}{\pi}} + O(1) \, ,
\end{eqnarray}
which implies an asymptotic scaling law with the network size $K$ for finite synaptic weight $\mu$:
\begin{eqnarray}
\beta \sim \sqrt{\frac{2 K \mu \beta}{\pi}} \quad \mathrm{i.e.} \quad  \beta \sim \frac{2 K \mu}{\pi} \, .
\end{eqnarray}
\end{remark}


\subsubsection{The relaxing model case}
The arguments proving \Cref{th:counting} for the counting-neuron model essentially
generalize to the RMF \emph{ansatz} for heterogeneous LGL networks with relaxation
(see \Cref{def:relaxingAnsatz}), albeit with some caveats.
Indeed, we show that the RMF \emph{ansatz} reduces to a set of self-consistency
equations by writing down that normalization conditions for the set of continuous solutions to the \emph{ansatz}.
We also show that continuous solutions are necessarily completely monotone,
which implies by Bernstein's theorem \cite{Feller2:1971}, that such solutions are indeed MGF for some probability distributions.
Moreover, utilizing monotonicity arguments, we show that \Cref{th:counting} implies
the existence of a solution $\bm{\beta}$ to the obtained set of self-consistency equations.
The main caveat is that we do not have any direct argument establishing the uniqueness of
solutions, although we conjecture that uniqueness holds for heterogeneous LGL networks with relaxation.
These results are summarized in the following theorem, which is proved in \Cref{sec:relaxing}:


\begin{theorem}\label{th:relaxingNeuron}
For all LGL relaxing models,
there is a set of independent real random variables
$\lbrace \Lambda_i \rbrace_{1 \leq i \leq K}$ whose MGFs
$\lbrace L_i \rbrace_{1 \leq i \leq K}$ are solutions to the RMF \emph{ansatz}  specified in \Cref{def:relaxingAnsatz} with
\begin{eqnarray}\label{eq:solMGF}
L_i(u) = \beta_i \int_{-\infty}^u \exp{ \left(\left[h_i(x)+\sum_{j \neq i} \beta_j h_{ij}(x) \right]^u_v +l_i(v)\right)} \, dv \, ,
\end{eqnarray}
where the functions $g_i$, $h_i$, and $h_{ij}$ are defined by
\begin{eqnarray}\label{eq:sysMIh1}
&l_i(x) = \tau_i r_i \left(e^{\frac{x}{\tau_i}} -1 \right) \, ,  \quad h_i(x) =  b_i \left( \tau_i \left(e^{\frac{x}{\tau_i}} -1 \right) - x \right) \, ,&\\ 
&h_{ij}(x) = \tau_i e^{-\tau_i \mu_{ij}} \left(\mathrm{Ei}\left(\tau_i \mu_{ij}e^{\frac{x}{\tau_i}} \right)-\mathrm{Ei}\left(\tau_i \mu_{ij} \right)\right) - x \, , \label{eq:sysMIh2}&
\end{eqnarray}
and where $\mathrm{Ei}$ denotes the exponential integral function.
In particular, the mean intensities $\Exp{\Lambda_i} = \beta_i$, $1 \leq i \leq k$, solve the system of equations
\begin{eqnarray}\label{eq:sysMI}
 \frac{1}{\beta_i} = \int_{-\infty}^0 \exp{\left(-h_i(v)-\sum_{j \neq i} \beta_j h_{ij}(v) +l_i(v) \right)}   \, dv \, . \\ \nonumber
\end{eqnarray}
\end{theorem}




\begin{remark}
The RMF \emph{ansatz} for neurons with excitatory random interaction weights and random reset values
takes the same form as in \Cref{def:relaxingAnsatz}:
\begin{eqnarray}
-\left( 1+ \frac{u}{\tau_i}\right)  L_i'(u)  + f_i(u) L_i(u) +  g_i(u) = 0 \, .
\end{eqnarray}
but with the functions
\begin{eqnarray}
f_i(u) &=& - \frac{ub_i}{\tau_i}  +\sum_{j \neq i}\left( 1-\int_0^\infty e^{u \mu} \, dq_{ij}(\mu)\right) \beta_j  \, ,  \\ \quad g_i(u)&=&\beta_i \int_0^\infty  e^{u r} \, dq_i(r) \, ,
\end{eqnarray}
where $q_{ij}$ is the probability measure of synaptic weight $\mu_{ij}$ and $q_i$
is the probability measure of the reset $r_i$. 
The above functions $f_i$ and $g_i$  still
satisfy the key properties (see Proposition \ref{prop:analysis2}) establishing \Cref{th:relaxingNeuron}, which therefore
extends straightforwardly to the case of excitatory random interactions and random reset values.
\end{remark}


\begin{remark}
The system of equations \eqref{eq:sysMI} can be interpreted probabilistically by considering an isolated relaxing-neuron $i$ subjected to independent Poissonian deliveries from other neurons with rate $\beta_j$.
Actually, one can check that the spiking activity of such a neuron defines a renewal process with a renewal distribution that satisfies
\begin{eqnarray}
\Prob{S_i > t} = \exp{\left(-h_i(-t)-\sum_{j \neq i} \beta_j h_{ij}(-t) +l_i(-t) \right)}  \, .
\end{eqnarray}
Then, the set of self-consistency equations \eqref{eq:sysMI} follows from writing:
\begin{eqnarray}
 \frac{1}{\beta_i} = \Exp{S_i } = \int_0^\infty \Prob{S_i > t} \, dt \, .
\end{eqnarray}
\end{remark}


\begin{remark}\label{rem:csm}
In the absence of relaxation, the inhomogeneous model becomes the ``counting-synapse model'', for which the stochastic intensities can be written as $\lambda_i(t) = b_i + \sum_{j \neq i} \mu_{ij} C_{ij}(t)$ via the introduction of the processes
\begin{eqnarray}
C_{ij}(t) = \int_{T_{i,0}(t)}^t N_j(ds) \,  , \quad j \neq i \, ,
\end{eqnarray}
which count the number of spikes that a neuron $i$ receives from anther neuron $j$ since the last time neuron $i$ spiked.
Taking the limit $\tau_i \to \infty$ in  \eqref{eq:sysMIh1} and \eqref{eq:sysMIh2} yields to the functions $g_i$, $h_i$, and $h_{ij}$ for the counting-synapses model
\begin{eqnarray}
l_i(x) = r_i x \, , \quad h_i(x) = 0 \, , \quad \mathrm{and} \quad h_{ij}(x) = \frac{e^{\mu_{ij}x}-1}{\mu_{ij}} - x \, ,
\end{eqnarray}
where the reset value $r_i$ coincides with the base rate ($r_i=b_i$).
\end{remark}

\section{Neuroscience applications}
\label{sec:impli}
The aim of this section is to illustrate the concrete applications of
the RMF approach through a few examples in neuroscience.
Since the main tool currently used for this class of problems is the TMF  limit,
we first compare the TMF  and the RMF models on a few basic network topologies and show how the latter outperforms the former.
A fundamental difference between the TMF  and the RMF is then discussed through the analysis of the so-called transfer functions of the two models.

\subsection{Numerical comparison with the thermodynamic limit}
At the core of the RMF approach is the assumption that the dynamics of finite-size LGL networks
is well-approximated by neurons experiencing independent Poissonian bombardments from other neurons.
As already mentioned, another possible simplifying assumption is that of the classical TMF  limit. 
In the TMF  model, one substitutes an individual neuron $i$ with a population of $M$ exchangeable
neurons with connections weights $\mu_{ji}/M$, and takes the limit of infinite population size $M \to \infty$.
Propagation of chaos holds in the TMF  limit \cite{DeMasi:2015}.
Thus, a neuron within population $i$ experiences neuronal interactions via the time-dependent deterministic drive 
\begin{eqnarray}
\alpha_i(t) =   \sum_{j \neq i} \mu_{ij} \left( \int_0^\infty \lambda p_j(t,\lambda) \, d\lambda \right) \, ,
\end{eqnarray}
where $p_j(t,\lambda)$ is the probability distribution of the stochastic intensity $\lambda$ of a
neuron within population $j$ at time $t$.
As a result, all neurons become independent in the TMF  limit, and each time-dependent probability
distribution $p_i$ satisfy a forward Kolmogorov equation that can be written
\begin{eqnarray}
\lefteqn{
\partial_i p_i(t,\lambda_i)
=
- \partial_{\lambda_i} \left[  \left(\frac{b_i-\lambda_i}{\tau_i}+ \alpha_i(t) \right)p_i(t,\lambda_i) \right]-} \nonumber\\
&& \hspace{100pt}  \lambda_i p_i(t,\lambda_i) + \left( \int_0^\infty \lambda p_i(t,\lambda) \, d\lambda \right) \delta_{r_i} (\lambda_i) \, . 
\end{eqnarray}
In the above right-hand side, the first term represents the deterministic drift incorporating
relaxation and interaction contributions, the second term is a death term due to neuronal spiking
with rate $\lambda_i$, and the last term represents a birth term localized at reset value $r_i$
with population-level rate $\int_0^\infty \lambda p_i(t,\lambda) \, d\lambda$.
Introducing the variables $s_i= b_i+\tau_i \sum_{j \neq i} \mu_{ij} \beta_j$,
the stationary distribution $p_i$ is thus solution to the equation
\begin{eqnarray}
 \partial_{\lambda_i} \left[ \left( \frac{s_i-\lambda_i}{\tau_i}  \right) p(\lambda_i) \right] + \lambda_i p_i( \lambda_i) 
= \beta_i \delta_{r_i}(\lambda_i) \, .
\end{eqnarray}
The stationary distribution solving the above equation can be expressed in closed form as
\begin{eqnarray}
p_i(\lambda) =  \frac{e^{\tau_i (\lambda-r_i)}}{\vert s_i-\lambda \vert} \bigg \vert \frac{s_i - \lambda}{s_i-r_i } \bigg \vert ^{\tau_i s_i} \beta_i \tau_i \, \mathbbm{1}_{[r_i,s_i]}(\lambda) \, ,
\end{eqnarray}
where $\mathbbm{1}_{[r_i,s_i]}$ is the indicator function of the interval $[r_i,s_i]$.
In turn, the MGF associated to the stationary distribution $p_i$ can be evaluated as
\begin{eqnarray}\label{eq:TMF MGF}
L_i(u) &=& \int e^{u \lambda} p_i(\lambda) \, d\lambda =
 \frac{\beta_i \tau_ie^{s_i u +(s_i-r_i) \tau_i }\gamma \big(\tau_i s_i, (s_i-r_i)(\tau_i+u) \big)}{\big( (s_i-r_i)(\tau_i+u)\big)^{\tau_i s_i}}  \,  \, ,
\end{eqnarray}
from which we deduce the set of TMF  self-consistency equations from the normalization conditions $L_i(0)=1$:
\begin{eqnarray}\label{eq:sysTMF }
\frac{1}{\beta_i} = \frac{\tau_i e^{(s_i-r_i)\tau_i}\gamma \big(\tau_i s_i, (s_i-r_i)\tau_i \big)}{\big( (s_i-r_i)\tau_i\big)^{\tau_i s_i}} \, .
\end{eqnarray}
Observe that the above self-consistency equations closely mirror the form of the set of
equations \eqref{eq:sysMI} obtained from the RMF \emph{ansatz}. 

\begin{figure}
  \centering
  \includegraphics[width=\textwidth]{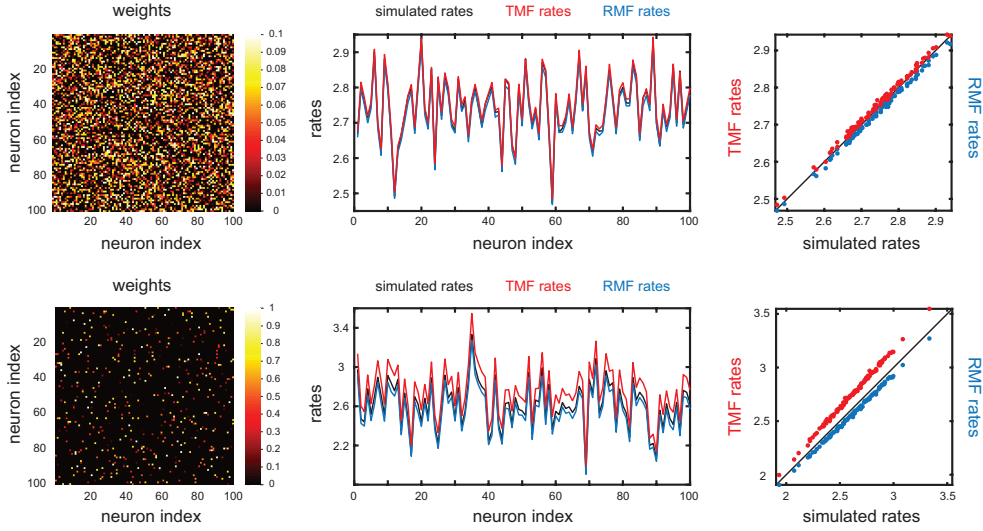}
  \caption{{\bf Recurrent network.} RMF models better capture the stationary firing rate of finite LGN networks than TMF  models for unstructured random networks with sparse, large, synaptic weights.
  {\bf Top row:} LGL network of $100$ counting-synapse neurons, each receiving spikes from randomly sampled $50$ neurons, via identically uniformly distributed synaptic weights. 
  {\bf Bottom row:}  LGL network of $100$ neurons, each receiving spikes from $5$ randomly sampled neurons, via identically uniformly distributed synaptic weights.
  {\bf Left:} Synaptic structure. {\bf Middle:} Numerical stationary rates obtained from discrete-event simulations ($10^7$ spiking events) and from iterated schemes for the RMF model and TMF  model ($20$ iterations). {\bf Right:} Scatter plots comparing the faithfulness of the TMF  model and that of the RMF model.
  }
   \label{fig:ERgraph}
\end{figure}

To explore the formal correspondence between the RMF and TMF  frameworks, let us  consider RMF models in the thermodynamic limit.
In considering such a limit, our goal is to evidence how TMF  models and first-order RMF models differ.
Applying the RCP to networks where we substitute each neuron with a population of $M$
exchangeable neurons yields the following RMF \emph{ansatz}:
\begin{eqnarray}
-\left( 1+ \frac{u}{\tau_i}\right)  L_i'(u) +  \left( \frac{ub_i}{\tau_i}  +\sum_{j \neq i}M \left( e^{\frac{u \mu_{ij}}{M}}-1\right) \beta_j \right) L_i(u) +  \beta_i e^{u r_i} = 0 \, .
\end{eqnarray}
Taking the thermodynamic limit, one has
$\lim_{M \to \infty} M \left( \exp{\left(u \mu_{ij}/M\right)}-1\right) = u \mu_{ij}$
and we obtain the new \emph{ansatz}
\begin{eqnarray}
-\left( 1+ \frac{u}{\tau_i}\right)  L_i'(u) +  \frac{us_i}{\tau_i}   L_i(u) +  \beta_i e^{u r_i} = 0 \, .
\end{eqnarray}
We refer to the above system of equations as the TMF  \emph{ansatz}.
As expected, one can check that the MGFs defined by relation \eqref{eq:TMF MGF} are solutions to the TMF  \emph{ansatz}.
Moreover, the difference between TMF  models and first-order MGF effectively appears to be due to the
terms mediating interactions: these terms are exponential in the first-order RMF limit, whereas they linearize in the TMF  limit.

\begin{figure}
  \centering
  \includegraphics[width=\textwidth]{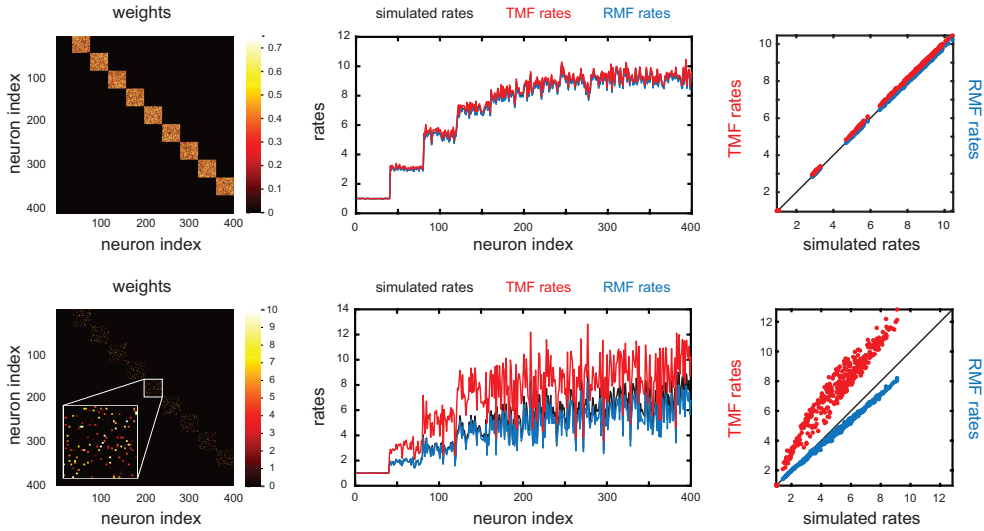}
  \caption{{\bf Feedforward network.}
  RMF models better capture the stationary firing rate of finite LGN networks than TMF  models for multilayered feedforward network with sparse, large, synaptic weights.
  {\bf Top row:} LGL network of $10$ layers of $40$ counting-synapse neurons, each receiving spikes from $40$ randomly sampled neurons from the previous layer (except the driving layer), via identically uniformly distributed synaptic weights. 
  {\bf Bottom row:}  LGL network of $10$ layers of $40$ neurons, each receiving spikes from $3$ randomly sampled neurons from the previous layer (except the driving layer), via identically uniformly distributed synaptic weights.
  {\bf Left:} Synaptic structure. {\bf Middle:} Numerical stationary rates obtained from discrete-event simulations ($10^7$ spiking events) and from iterated schemes for the RMF model and TMF  model ($20$ iterations). {\bf Right:} Scatter plots comparing the faithfulness of the TMF  model and that of the RMF model.
  }
  \label{fig:SFgraph}
\end{figure}

Moreover, we present numerical results emphasizing when the first-order RMF approach approximates
finite LGL networks more faithfully than TMF  networks.
We consider two types of counting-synapse models (see \Cref{rem:csm}): unstructured recurrent networks
in \Cref{fig:ERgraph} and multilayered feedforward networks in \Cref{fig:SFgraph}.
For each network structure, we numerically evaluate the empirical stationary firing rates
of finite LGL networks via discrete-event simulations using the Gillespie algorithm \cite{Gillespie:1977}.
Then, we compare these empirical rates with the RMF rates and the TMF  rates,
which are obtained by numerically solving the self-consistency equations \eqref{eq:sysMI} and \eqref{eq:sysTMF }, respectively.
These solutions are computed via the---empirically unconditionally converging--- iteration scheme
deduced from the self-consistency equations.
As expected from our discussion of the TMF  limit, \Cref{fig:ERgraph}a and \Cref{fig:SFgraph}a
show that RMF models closely mirror TMF  models for LGL networks with weak interactions. e.g., with $\mu_{ij} / b_i \ll 1$.
Moreover, TMF  models, as well as RMF models, are both faithful approximations of 
the corresponding finite LGL networks, which exhibit weak correlations by construction.
Because of the role played by the interaction-mediating terms in the TMF  and RMF \emph{ans\"atze},
we expect that RMF models become distinct from TMF  models for network structure involving large synaptic weights,
e.g., with $\mu_{ij}/b_i > 1$. However, we expect RMF model to be faithful only when the Poisson Hypothesis
is a good modeling assumption, i.e., when spike trains are nearly Poissonian and independent across neurons.
For large synaptic weights, such a behavior is the hallmark of sparsely connected networks.
\Cref{fig:ERgraph}b and \Cref{fig:SFgraph}b confirm that RMF networks better predict the firing
rates of LGL networks with large, sparse, synaptic connections.
Further numerical simulations reveal that RMF models comparatively better capture
feedforward networks than recurrent networks (see \Cref{table:comp}). This is due to the presence of cycles in the network structure,
which promotes correlation and gradually invalidates the Poisson Hypothesis \cite{Melamed:1979}.
Accounting for networks with large, sparse, synaptic connections but strong recurrent structure, e.g.,
nearest-neighbor lattice graph, requires to consider higher-order RMF models  (see \Cref{sec:future}).

\begin{table}[ht]
\caption{Comparison of the relative errors of the mean firing rates in the TMF  limit and in the RMF limit for different network structures.
The RMF limit comparatively better captures the mean firing rates for LGL networks with large, sparse, synaptic connections.}
\begin{center}
  \begin{tabular}{|c|c|c|}
    \hline
  Network model & TMF  error & RMF error \\ \hline
 Complete unstructured & $<1\%$ & $< 1\%$   \\ 
 Sparse unstructured  & $5\%$ & $2\%$ \\ \hline
 Complete feedforward  & $2\%$ & $1\%$ \\ 
 Sparse feedforward & $44\%$ & $7\%$ \\
    \hline
  \end{tabular}
\end{center}
\label{table:comp}
\end{table}

\begin{figure}
  \centering
  \includegraphics[width=\textwidth]{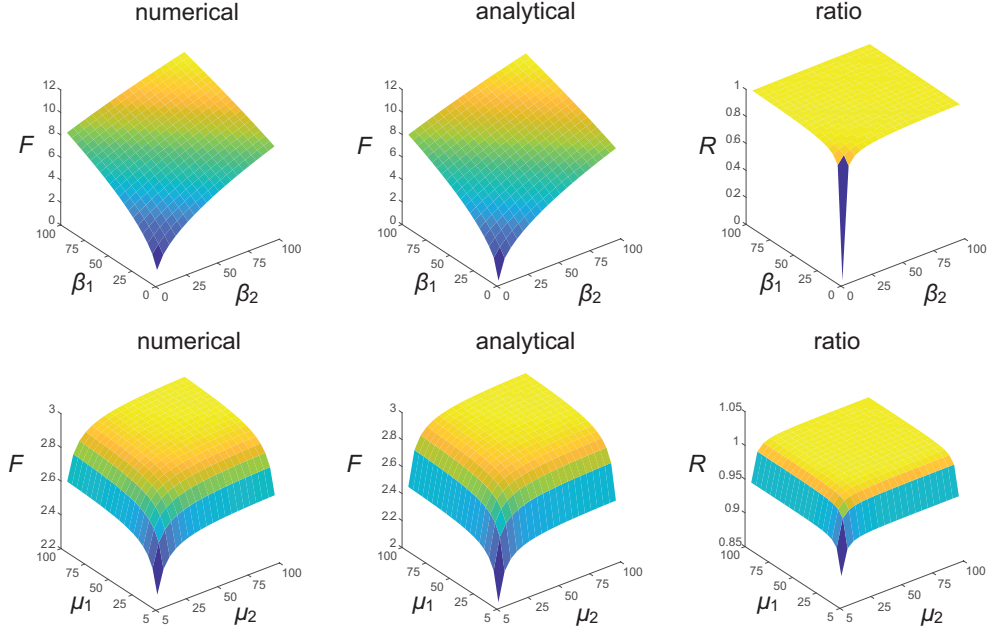}
\caption{{\bf Transfer function.} Asymptotic regime of the transfer function $F$ for a neuron
with reset value $r=1$, base level $b=1$, time constant $\tau=1$, and receiving spikes from two other neurons.  
{\bf Top row.} Numerical and analytical approximation of the transfer function $F$ for large input rates with
synaptic weights $\mu_1=\mu_2=1$. A purely excitatory LGL network is always stable because its transfer
function grows sublinearly as a function of its input rates. {\bf Bottom row.}  Numerical and analytical
approximation of the transfer function $F$ for large synaptic weights with input rates $\beta_1=\beta_2=1$.
The transfer function saturates for large synaptic weights showing the non-symmetric role of synaptic weights and input rates.}
   \label{fig:Transfer}
\end{figure}


\subsection{Asymptotic transfer functions} 
A key quantity determining the behavior of neural networks is the neuronal rate-transfer function,
which relates the output stationary rate of a neuron to its stationary input rates and its synaptic weights. 
For instance, neurons modeled via Hawkes processes---which neglect reset mechanisms---have rate-transfer functions
that depend linearly on the rates of interaction $\mu_{ij} \beta_j$.
Such a linear dependence of rate-transfer functions implies that Hawkes neural networks are prone
to explosion in the absence of inhibition, and thus fail to admit a stationary regime.
By contrast, LGL networks are unconditionally stable, indicating that the LGL rate-transfer
function must grow sublinearly with input rates.
Within the RMF framework, the rate-transfer function of a neuron $i$, denoted $F_i$, is given by the self-consistency
equations \eqref{eq:sysMI} and can be expressed as
\begin{eqnarray}\label{eq:transfer}
F_i({\bm \beta},{\bm \mu})= \left(\int_{0}^{\tau_i} \exp{\left(-\tilde{h}_i(v)-\sum_{j \neq i} \beta_j \tilde{h}_{ij}(v) \right)}  \tilde{l}_i(v)   \, dv \right)^{-1}  \, ,  \nonumber
\end{eqnarray}
where the auxiliary functions $\tilde{h}_{ij}$, $\tilde{h}_{i}$ and $\tilde{l}_{i}$ are defined as:
\begin{eqnarray}
\hspace{25pt} \tilde{h}_{ij}(v) = \int_0^v  \frac{ 1-e^{-\mu_{ij}u}}{1-u/\tau_i}  \, du \, , \quad \tilde{h}_{i}(v) = \frac{b_i}{\tau_i}\int_0^v  \frac{u}{1-u/\tau_i}  \, du \, , \quad \tilde{l}_{i}(v) = \frac{e^{-r_i v}}{1-v/\tau_i} \, 
\end{eqnarray}
(see Equation \ref{eq:usedeq}).
In \Cref{fig:Transfer}, we numerically compute the rate-transfer function of a neuron subjected to two spiking streams with varying input rates and varying synaptic weights.
Considering the asymptotic behavior of $F_i$ via the Laplace method in the limit of large input rates $\beta_j$ exhibits the sublinearity of $F_i$.
Specifically, observing that the function $\tilde{h}_{ij}$ admits its minimum over $(0,\tau_i)$ in $0$, the Laplace method implies that in the limit of large input rates, i.e., for all $\beta_j \to \infty$, we have
\begin{eqnarray}
F_i({\bm \beta},{\bm \mu})^{-1} \sim  e^{-\tilde{h}_i(0)-\sum_{j \neq i} \beta_j \tilde{h}_{ij}(0) } \tilde{l}_i(0) \int_0^{\infty} e^{-\sum_{j \neq i} \beta_j \tilde{h}''_{ij}(0)v^2/2 } dv \, .
\end{eqnarray}
The evaluation of the Gaussian integral with $\tilde{h}''_{ij}(0)=\mu_{ij}$ yields the asymptotic behavior
\begin{eqnarray}
F_i({\bm \beta},{\bm \mu}) = \bigg(\frac{2}{\pi} \sum_{j \neq i} \mu_{ji} \beta_j \bigg)^{1/2} + o\left(  \sqrt{\beta_1} , \ldots,   \sqrt{\beta_K} \right) \, ,
\end{eqnarray}
showing that LGL rate-transfer functions scale with the square-root of the input rates, which is consistent with the reset-enforced unconditional stability of LGL networks. 
Such a sublinear scaling is the same as that of the counting-neuron model because relaxation becomes irrelevant at high firing rate, i.e., when interspike intervals become shorter than the relaxation time constant $\tau_i$ (see \Cref{fig:Transfer}).

Finally, by contrast with Hawkes model---and with LGL neurons in the TMF  limit---, the rate-transfer function $F_i$ exhibits a distinct nonlinear dependence on the synaptic weights at fixed input rates.
Indeed,  we have
\begin{eqnarray}
\tilde{h}_{ij}(v) =  -\sum_{j \neq i} \frac{1}{\mu_{ij}} + O\left( 1/ \mu_1^2, \ldots, 1/ \mu_K^2\right) \, ,
\end{eqnarray}
Then, taking the limit $\mu_{ij}\to \infty$ in \eqref{eq:transfer} shows that the rate-transfer function  $F_i$ asymptotically saturates to the upper bound 
\begin{eqnarray}
\bar{\beta}_i = \frac{e^{-a} a^b }{\tau_i \gamma (b,a)} \quad \mathrm{with} \quad a=\tau_i(b_i-r_i) \quad \mathrm{and} \quad b=\tau_i \left( b_i + \sum_{j \neq i} \beta_j\right) \, .
\end{eqnarray}
This upper bound simplifies to $\bar{\beta}_i = b_i + \sum_{j \neq i} \beta_j$ when the reset level and the base level identical: $b_i=r_i$.
Finally, accounting for first-order corrections shows that for large synaptic weights, we have the scaling
\begin{eqnarray}
F_i({\bm \beta},{\bm \mu})^{-1} \sim \int_{0}^{\tau_i} e^{\frac{av}{\tau_i}+\sum_{j \neq i} \frac{\beta_j}{\mu_{ij}}} \Big( 1-v/\tau_i\Big)^{b-1}  \, dv =
 e^{\sum_{j \neq i} \frac{\beta_j}{\mu_{ij}}} / \bar{\beta}_i  \, ,
\end{eqnarray}
so that the rate-transfer function $F_i$ has the following asymptotic behavior
\begin{eqnarray}
F_i({\bm \beta},{\bm \mu}) = \bar{\beta}_i \left(1- \sum_{j \neq i} \frac{\beta_j}{\mu_{ij}}\right) + o\left( 1/ \mu_1, \ldots, 1/ \mu_K \right) \, .
\end{eqnarray}
This saturating behavior is a distinct feature of RMF limit models (see \Cref{fig:Transfer}).
Informally, in the limit of infinite weights, each spiking input triggers a spiking output leading to an effective quasi-linear transfer function.
By contrast, in the TMF  limit, increasing synaptic weight $\mu_{ij}$ is equivalent to increasing input rate $\beta_j$, so that the rate-transfer function diverges in the limit of large synaptic weights. 
This failure to capture saturation in the TMF  limit explains why RMF models outperforms TMF models for sparse networks with large synaptic weights.


\section{Proofs}\label{sec:proof}

This section contains the proofs of the key results of our RMF computational framework.
\Cref{sec:markov} contains the Markovian analysis justifying the Harris ergodicity of LGL networks
and their finite replica versions (\Cref{sec:Harris}) and the derivation of the RMF \emph{ansatz} (\Cref{sec:funcRepModel}).
\Cref{sec:relaxing} proves \Cref{th:relaxingNeuron} solving the RMF \emph{ansatz}
for the relaxing-neuron model with synaptic heterogeneity.


\subsection{Markovian analysis}\label{sec:markov}

Establishing Harris ergodicity, as well as deriving the RMF \emph{ansatz},
essentially rely on the Markovian analysis of the infinitesimal generators of LGL networks and their finite replica versions.


\subsubsection{Harris ergodicity}\label{sec:Harris}

To prove Harris ergodicity, it is enough to exhibit a regeneration set that is positive recurrent
for $\{\bm{\Lambda}_n \}_{n \in \mathbb{Z}}$, the embedded Markov chain of $\{ {\bm \lambda}(t) \}_{t \in \mathbb{R}}$,
defined as $\{\bm{\Lambda}_n \}_{n \in \mathbb{Z}} = \{\bm{\lambda}_{T_n} \}_{n \in \mathbb{Z}}$,
where $T_n$ denotes the ordered sequence of jumps such that almost surely $T_0 \leq 0 < T_1$ and $T_n<T_{n+1}$. 
In \cite{Robert:2016}, Robert and Touboul exploit the Poissonian embedding of intensity-based network models
\cite{Kerstan:1964} to show that all compact sets $R_{\lambda_0} = [0,\lambda_0]^K$ with 
\begin{eqnarray}
\lambda_0> \max_i \left( \sum_j \mu_{ji} + b_i \right) \, ,
\end{eqnarray}
are regeneration sets for $\{\bm{\Lambda}_n \}_{n \in \mathbb{Z}}$.
Briefly, regeneration happens when each neuron spikes consecutively and ``spontaneously'', i.e.,
in the base-rate component of the Poissonian embedding, which is well defined as long as $\min_i \inf_t \lambda_i(t) = \min_i r_i>0$. 
Given an initial state ${\bf{\Lambda}}_{0}$ in $R_{\lambda_0}$, such a sequence of $K$ transitions yields
a state ${\bf{\Lambda}}_{K}$ that is independent of ${\bf{\Lambda}}_{0}$, while happening with finite, albeit small, probability. 
Knowing the regenerative property of compact sets $R_{\lambda_0}$, the Harris ergodicity of
$\{ {\bm \lambda}(t) \}_{t \in \mathbb{R}}$ follows from the existence of positive recurrent compact sets
under the assumption of a non-explosive behavior. The non-explosive nature of the dynamics, as well as the positive
recurrence of compact sets $R_{\lambda_0}$ for large enough $\lambda_0$, are established by verifying the
following Foster-Lyapunov drift condition for exponential scale functions
$V_u({\bm \lambda}) = \exp{ \left(u \sum_i \lambda_i \right)}$:


\begin{proposition}\label{th:FosterLyapunov}
For $u>0$ and $c>0$, there are real numbers $d>0$ and $l>0$ such that for all
$\lambda_0>l$ and for all $\lambda$ in $\mathbb{R}_+^K$, we have
\begin{eqnarray}
\mathcal{A}[V_u]({\bm \lambda}) \leq -c V({\bm \lambda}) + d\mathbbm{1}_{R_{\lambda_0}}({\bm \lambda}) \, .
\end{eqnarray}
\end{proposition}


\begin{proof}
On $\mathbb{R}_+^K$, the infinitesimal increment of the scale function $V_u$ satisfies
\begin{eqnarray}
\mathcal{A}[V_u](\bm{\lambda}) &=&  \sum_i  \frac{b_i - \lambda_i}{\tau_i} \, u V_u(\bm{\lambda}) +  \sum_i \left( e^{u\left(\sum_{j \neq i} \mu_{ji} + r_i-\lambda_i\right)}-1 \right) \lambda_i V_u(\bm{\lambda}) \, , \\
&\leq &
\left(  u \sum_i  \frac{b_i}{\tau_i}  +  \frac{1}{u} \sum_i e^{u\left(\sum_{j \neq i} \mu_{ji} + r_i\right)-1}- \sum_i \lambda_i \right)V_u(\bm{\lambda}) \, .
\end{eqnarray}
where we used the facts that $\lambda_i \geq 0$ and that $\max_\lambda e^{-\lambda_i u}\lambda = 1/u e$ for $u>0$.
Given $c>0$, the compact set 
\begin{eqnarray}
R_c = \left\{ {\bm \lambda} \in \mathbb{R}^K_+ \, \big \vert  \, \sum_i \lambda_i \leq u \sum_i  \frac{b_i}{\tau_i}  +  \frac{1}{u} \sum_i e^{u\left(\sum_{j \neq i} \mu_{ji} + r_i\right)-1} + c \right\}
\end{eqnarray}
is such that $\mathcal{A}[V_u] \leq -c V_u$ outside $R_c$. 
Thus, choosing
\begin{eqnarray}
l = u \sum_i  \frac{b_i}{\tau_i}  +  \frac{1}{u} \sum_i e^{u\left(\sum_{j \neq i} \mu_{ji} + r_i\right)-1} + c \, , 
\end{eqnarray}
implies that, for $\lambda_0> l$, $\mathcal{A}[V_u] \leq -c V_u$ outside of $R=[0,\lambda_0]^K \supset R_c$. 
Moreover, using the boundedness of $V_u$ on compact sets to choose
\begin{eqnarray}
d =\lambda_0  \sup_{\bm{\lambda} \in R}V_u(\bm{\lambda}) < \infty \, ,
\end{eqnarray}
we finally check that $\mathcal{A}[V_u] \leq -c V_u + d\mathbbm{1}_{R}$ on $\mathbb{R}_+^K$.
\end{proof}

In \cite{MeynTweedie:1993}, Meyn and Tweedie show that the Foster-Lyapunov drift condition of Proposition \ref{th:FosterLyapunov} has two immediate implications:
$i)$ As the functions $V_u$ are positive and norm-like, i.e. $\lim_{{\bm \lambda} \to \infty} V_u({\bm \lambda}) = \infty$ for $u>0$, Proposition \ref{th:FosterLyapunov} directly implies that the Markovian dynamics is non-explosive.
$ii)$ As the dynamics is non-explosive and noting that $V_u \geq 1$ on $\mathbb{R}^K_+$, a set $R_{\lambda_0}$ satisfying Proposition \ref{th:FosterLyapunov} is positive recurrent, and for large enough $\lambda_0$, $R_{\lambda_0}$ is a regeneration set as well, implying the Harris ergodicity of the Markov chain $\{ \lambda(t) \}_{t \in \mathbb{R}}$.


\subsubsection{Functional equations for replica models}\label{sec:funcRepModel}

Following the exact same steps as for the proof of \Cref{th:exactPDE}, Dynkin's formula applied at stationarity allows one to functionally characterize the stationary state of the $M$-replica model as stated in \Cref{th:exactPDErep}.

\begin{proof}[Proof of \Cref{th:exactPDErep}]
Given a subset of replica indices $S \subset  \{1, \ldots, M\}$, let us express the infinitesimal generator $\mathcal{A}$ defined by expression \eqref{eq:infGen2} for the $M$-replica model when acting on the exponential function
\begin{eqnarray}
f_{\bm{u}}(\bm{\lambda}) = \exp{\left(\sum_{i=1}^{K} \sum_{m \in S} u_{m,i} \lambda_{m,i}\right)} \, .
\end{eqnarray}
We obtain the relation
\begin{eqnarray}
\mathcal{A}[f_{\bm{u}}](\bm{\lambda}) 
&=& \sum_{i=1}^{K} \sum_{m \in S} \left( \frac{b_i - \lambda_{m,i}}{\tau_i} \right) u_{m,i} f_{\bm{u}}(\bm{\lambda}) \nonumber\\
&+& \sum_{i=1}^{K} \sum_{m \in S}  \frac{1}{\vert V_{m,i} \vert} \sum_{\bm{v} \in V_{m,i}} \left( e^{u_{m,i}(r_i-\lambda_{m,i}) + \sum_{j \neq i, v_j \in S} u_{v_j\!,j} \mu_{ji} }- 1 \right) f_{\bm{u}}(\bm{\lambda}) 
\lambda_{m,i} \nonumber\\
&+& \sum_{i=1}^{K} \sum_{m \notin S}  \frac{1}{\vert V_{m,i} \vert} \sum_{\bm{v} \in V_{m,i}}\left( e^{\sum_{j \neq i, v_j \in S} u_{v_j\!,j} \mu_{ji} } - 1 \right) f_{\bm{u}}(\bm{\lambda}) \lambda_{m,i} \, . 
\end{eqnarray}
By Dynkin's formula, we have $\Exp{\mathcal{A}[f_{\bm{u}}](\bm{\lambda}) } = 0$ for stationary $M$-replica dynamics, which implies that
\begin{eqnarray}\label{eq:replicaS}
0
&=& \sum_{i=1}^{K} \sum_{m \in S} \left( \frac{b_i u_{m,i}}{\tau_i} L(\bm{u}) -  \frac{u_{m,i}}{\tau_i}  \partial_{u_{m,i}}L(\bm{u}) \right) \nonumber\\
&+& \sum_{i=1}^{K} \sum_{m \in S}  \frac{1}{\vert V_{m,i} \vert} \sum_{\bm{v} \in V_{m,i}} \left( e^{\left( u_{m,i}r_i + \sum_{j \neq i, v_j \in S} u_{v_j\!,j} \mu_{ji}  )\right)}- 1 \right)  \partial_{u_{m,i}}L(\bm{u}) \big \vert_{u_{m,i}=0} \nonumber\\
&&\hspace{1cm}+ \sum_{i=1}^{K} \sum_{m \notin S}  \frac{1}{\vert V_{m,i} \vert} \sum_{\bm{v} \in V_{m,i}}\left( e^{\left(\sum_{j \neq i, v_j \in S} u_{v_j\!,j} \mu_{ji} \right)} - 1 \right) \Exp{\lambda_{m,i} f_{\bm{u}}(\bm{\lambda})}\ ,
\end{eqnarray}
where we use the notation 
\begin{eqnarray}
L(\bm{u}) = \Exp{\exp{\left(\sum_{i=1}^{K} \sum_{m \in S} u_{m,i} \lambda_{m,i}\right)}} \, .
\end{eqnarray}
Specifying the above relation for $S= \{1, \ldots, M\}$ yields the PDE of \Cref{th:exactPDErep}.
\end{proof}


In the remaining of this section, we justify relation \eqref{eq:prepAnsatz} used for heuristically
deriving the RMF \emph{ansatz} of \Cref{def:relaxingAnsatz}.
Considering only the first replica $S=\lbrace 1\rbrace$, and denoting $u_{1,j}=u_j$ for simplicity,
relation \eqref{eq:replicaS} becomes
\begin{eqnarray}
0
&=& \sum_{i=1}^{K}  \left( \frac{b_i u_{i}}{\tau_i} L(\bm{u}) -  \frac{u_{i}}{\tau_i}  \partial_{u_{i}}L(\bm{u}) \right) \\
&+& \sum_{i=1}^{K}  \frac{1}{\vert V_{i,1} \vert} \sum_{\bm{v} \in V_{i,1}} \left( e^{\left(\sum_{j \neq i, v_j =1} u_{j} \mu_{ji}  + u_{i}r_i \right)} \!-\! 1 \right)  \partial_{u_{i}}L(\bm{u}) \big \vert_{u_{i}=0} \nonumber \\
&+& \sum_{i=1}^{K} \sum_{m >1}  \frac{1}{\vert V_{m,i} \vert} \sum_{\bm{v} \in V_{m,i}}\left( e^{\left(\sum_{j \neq i, v_j =1} u_{j} \mu_{ji} \right)} \!-\! 1 \right) \Exp{\lambda_{m,i} f_{\bm{u}}(\bm{\lambda})} \, .\nonumber
\end{eqnarray}
As $\bm{v} \in V_{i,1}$ implies $v_j \neq 1$ for all $j \neq i$, the exponent in the second term of the right-hand side is actually independent of $\bm{v}$ so that we have:
\begin{eqnarray}
0
&=& \sum_{i=1}^{K}  \left( \frac{b_i u_{i}}{\tau_i} L(\bm{u}) -  \frac{\lambda_{i}}{\tau_i}  \partial_{u_{i}}L(\bm{u}) \right) \\
&+& \sum_{i=1}^{K}   \left( e^{u_{i}r_i}- 1 \right)  \partial_{u_{i}}L(\bm{u}) \big \vert_{u_{i}=0} \nonumber\\
&+& \sum_{i=1}^{K} \sum_{m >1}  \frac{1}{\vert V_{m,i} \vert} \sum_{\bm{v} \in V_{m,i}}\left( e^{\left(\sum_{j \neq i, v_j =1} u_{j} \mu_{ji} \right)} \!-\! 1 \right) \Exp{\lambda_{m,i} f_{\bm{u}}(\bm{\lambda})} \, .\nonumber
\end{eqnarray}
By exchangeability of replicas, the value of the expectation term above is independent of $m>1$.
Then, conditionally to neuron $i$ spiking, let us estimate the sum:
\begin{eqnarray}
\sum_{m >1}  \frac{1}{\vert V_{m,i} \vert} \sum_{\bm{v} \in V_{m,i}}\left( e^{\left(\sum_{j \neq i, v_j =1}u_{j} \mu_{ji} \right)} - 1 \right)  
= 
\frac{(M-1)S_{i,2}}{\vert V_{i,2} \vert} \ ,
\end{eqnarray}
where $S_{i,2}$ collects the terms corresponding to interactions with the second replica:
\begin{eqnarray}
S_{i,2} = \sum_{\bm{v} \in V_{i,2}}\left( e^{\sum_{j \neq i, v_j =1}u_{j} \mu_{ji} } - 1 \right) \, .
\end{eqnarray}
To further estimate $S_{i,2}$, observe that the set $V_{i,2}$ can be partitioned according to how many of its components are equal to one.
Specifically, we have the partition
\begin{eqnarray}
V_{i,2} = V^{(0)}_{i,2} \cup \dots \cup V^{(K-1)}_{i,2} \, ,
\end{eqnarray}
where the non-overlapping sets $V^{(k)}_{i,2}$, $0 \leq k \leq K-1$, are defined as
\begin{eqnarray}
\hspace{20pt}
V^{(k)}_{i,2}
=
\big \lbrace  \bm{v} \in V_{i,2} \, \big \vert \,  \vert \lbrace v_j = 1 \rbrace \vert =k \big \rbrace  \quad \mathrm{with} \quad \Big \vert V^{(k)}_{i,2} \Big \vert =  \binom{K}{k} (M-2)^{K-1-k}\, .
\end{eqnarray}
Noticing that  $\exp{\left(\sum_{j \neq i, v_j =1}u_{j} \mu_{ji} \right)} - 1=0$ on $V^{(0)}_{i,2}$,
we have
\begin{eqnarray}
S_{i,2}
&=&
\sum_{k=1}^{K-1}  \sum_{\bm{v} \in V^{(k)}_{i,2}} \left( e^{\sum_{j \neq i, v_j =1}u_{j} \mu_{ji} } - 1 \right) \nonumber\\
&=&
(M-2)^{K-2}\sum_{j \neq i}  \left( e^{u_{j} \mu_{ji}} -1 \right) + (M-2)^{K-3}\sum_{j,k \neq i}  \left( e^{u_{j} \mu_{ji}+u_{k} \mu_{ki}} -1 \right) + \ldots
\end{eqnarray}
Remembering that $\vert V_{m,i} \vert = (M-1)^{K-1}$, we conclude that when $M \to \infty$, we have
\begin{eqnarray}
\hspace{5pt} \sum_{m >1}  \frac{1}{\vert V_{m,i} \vert} \sum_{\bm{v} \in V_{m,i}}\left( e^{\left(\sum_{j \neq i, v_j =1}u_{j} \mu_{ji} \right)} - 1 \right)  
= 
\sum_{j \neq i}  \left( e^{u_{j} \mu_{ji}} -1 \right) + O(1/M) \, ,
\end{eqnarray}
which justifies relation \eqref{eq:prepAnsatz} under assumption that the involved expectation terms remain bounded when $M \to \infty$.


\subsection{Solutions to the RMF \emph{ansatz}}\label{sec:relaxing}

Solving the RMF \emph{ansatz} for the relaxing-neu\-ron model with synaptic heterogeneity is more involved
than for the counting-neuron model.  This is primarily due to the fact that in the presence of relaxation, 
stochastic intensities have a continuous state space, which requires to consider MGFs instead of PGFs.
The defining property of MGFs is provided by the criterion of complete monotonicity.
To prove \Cref{th:relaxingNeuron}, we first show that the RMF \emph{ansatz} admits a unique smooth solution (\Cref{sec:uniquesmooth}).
Then, we show that this smooth solution is completely monotone (\Cref{sec:compmon}).
Finally, we show that the condition of normalization for smooth solutions reduces
to the announced set of equations for the mean neuronal intensities, which admits at least one solution (\Cref{sec:relaxsol}).


\subsubsection{Uniqueness of smooth solutions}\label{sec:uniquesmooth}

Just as for the counting-neuron model, there is a unique smooth solution to the type 
of ODEs intervening in the RMF {\emph ansatz} for the relaxing-neuron model with synaptic heterogeneity.
This is stated in the following proposition:

\begin{proposition}\label{prop:analysis1}
Let $f$ and $g$ be real-valued functions in $C^{n+1}(\mathbb{R})$ with $n \geq 1$ 
and $\tau$ a positive real number such that $f(-\tau)>0$, then the ODE
\begin{eqnarray}\label{eq:analysisODE}
\left( 1+ \frac{u}{\tau}\right) L'(u)  + f(u) L(u) -  g(u) = 0 \, ,
\end{eqnarray}
admits a unique continuous solution on $\mathbb{R}$:
\begin{eqnarray}\label{eq:solAnalytical}
L(u) =   \int_{-\tau}^u e^{-\int_v^u  \frac{f(w)}{1+w/\tau} \, dw}  \frac{g(v)}{1+v/\tau} \, dv \, .
\end{eqnarray}
Moreover, this solution admits a derivative of order $n$ in $-\tau$. In particular, we have
\begin{eqnarray}
L(-\tau) = g(-\tau)/f(-\tau) \quad \mathrm{with} \quad L'(-\tau) = \frac{(g/f)'(-\tau)}{1+(\tau f(-\tau))^{-1}} \, .
\end{eqnarray}
\end{proposition}


\begin{proof}
$i)$ \emph{Uniqueness.}
As $f$ and $g$ are continuous on $\mathbb{R}$,  \eqref{eq:sysLaplace} admits continuously
differentiable solutions on $(-\infty,-\tau)$ and $(-\tau,+\infty)$.
Solutions defined on $(-\tau,+\infty)$ have the generic integral expression 
\begin{eqnarray}\label{eq:solAnalysis}
L(u) = L_0 e^{-\int_0^u \frac{f(v)}{1+v/\tau} \, dv} +   \int_0^u e^{-\int_v^u  \frac{f(w)}{1+w/\tau} \, dw}  \frac{g(v)}{1+v/\tau} \, dv \, ,
\end{eqnarray}
where $L_0$ denotes the arbitrary real value taken by $L$ in zero.
The analysis of the above expression shows that solutions on $(-\tau,+\infty)$
generically have an infinite discontinuity when $u \to -\tau^+$.
In fact, we evaluate by integration by parts that
\begin{eqnarray}
\frac{1}{\tau} \int_0^u \frac{f(v)}{1+v/\tau} \, dv
= \left[ f(v)\ln{\left(1+\frac{v}{\tau} \right)}\right]_0^u - \int_0^u f'(v)\ln{\left(1+\frac{v}{\tau} \right)} \, dv\, ,
\end{eqnarray}
where the integral in the right-hand term has a finite limit when $u \to -\tau^+$.
Thus, the homogeneous part of $L$ exhibits the asymptotic behavior 
\begin{eqnarray}
e^{-\int_0^u \frac{f(v)}{1+v/\tau} \, dv} 
\sim 
c \left(1+ \frac{u}{\tau}\right)^{-a} \, , \quad u \to -\tau^+
\end{eqnarray}
where we have set the constants
\begin{eqnarray}
a=\tau f(-\tau)>0 \,  \quad \mathrm{and} \quad
c= - \tau  \int_{-\tau}^0 f'(v)\ln{\left(1+\frac{v}{\tau} \right)} \, dv \, ,
\end{eqnarray}
thereby showing that $L$ generically has an infinite discontinuity in $-\tau$. 
Factorizing the homogeneous part leads to considering $L$ under the form
 \begin{eqnarray}\label{eq:L0}
L(u) = e^{-\int_0^u \frac{f(v)}{1+v/\tau} \, dv} \left( L_0+  \int_0^u e^{-\int_v^0  \frac{f(w)}{1+w/\tau} \, dw}  \frac{g(v)}{1+v/\tau} \, dv  \right) \, .
\end{eqnarray}
For $L$ to have a finite left-limit in $-\tau$, the term in parentheses in the above expression
must vanish when $u \to -\tau^+$, which implies that one must choose 
\begin{eqnarray}
L_0 = \lim_{u \to -\tau^+} \int_u^0 e^{-\int_v^0  \frac{f(w)}{1+w/\tau} \, dw}  \frac{g(v)}{1+v/\tau} \, dv \, .
\end{eqnarray}
The above limit exists and is finite due to the asymptotic behavior of the integrand
\begin{eqnarray}
e^{-\int_v^0  \frac{f(w)}{1+w/\tau} \, dw}  \frac{g(v)}{1+v/\tau}
\sim
\frac{g(-\tau)}{c} \left(1+ \frac{v}{\tau}\right)^{a-1}  \, ,
\end{eqnarray}
where the right-hand term is integrable ($a>0$).
This shows that a continuous solution to  \eqref{eq:analysisODE} must take
a unique value $L_0$ in $0$ and is therefore uniquely characterized on $(-\tau, +\infty)$.
Moreover, inserting the integral expression for $L_0$ given by \eqref{eq:L0} into \eqref{eq:solAnalysis} 
yields the announced expression \eqref{eq:solAnalytical} for that unique solution.
Repeating the above analysis on $(-\infty,-\tau)$ rather than $(-\tau,+\infty)$ would yield
the same expression for the unique solution with a finite right-limit in $-\tau$,
showing that there is at most one continuous solution to \eqref{eq:solAnalysis} on $\mathbb{R}$.

$ii)$ \emph{Existence: continuity.} 
It is enough to show that the function $L$ defined on $\mathbb{R}\setminus \lbrace \tau \rbrace$
by \eqref{eq:solAnalytical} is continuous in $-\tau$. 
In order to compute $\lim_{u \to \tau} L(u)$, we first use integration by part to obtain
the asymptotic behavior of the exponent function in \eqref{eq:solAnalytical} when $u \to -\tau$:
\begin{eqnarray}
\frac{1}{\tau} \int_u^v \frac{f(w)}{1+v/\tau} \, dw
&=& \left[ f(w)\ln{\left(\Big \vert 1+\frac{w}{\tau}\Big \vert  \right)}\right]_u^v - \int_u^v f'(w)\ln{\left(\Big \vert 1+\frac{w}{\tau} \Big \vert \right)} \, dw \, ,\nonumber\\
&=& f(-\tau) \ln{\left(\Big \vert \frac{\tau+v}{\tau+u}\Big \vert \right)}+ o_{-\tau}(1) \, , \quad  \, \vert \tau+ v\vert < \vert \tau+ u \vert   \,  .
\end{eqnarray}
Thus we have the equivalence
\begin{eqnarray}
e^{-\int_v^u  \frac{f(w)}{1+w/\tau} \, dw} \sim \left(\frac{\tau+v}{\tau+u}\right)^a , \quad  \, 0 < \frac{\tau+v}{\tau+u}< 1 \, , \quad u \to -\tau \,  ,
\end{eqnarray}
which shows that the sought-after limit can be evaluated as:
\begin{eqnarray}
\lim_{u \to -\tau} L(u) =  \lim_{u \to -\tau} \int_{-\tau}^u \left(\frac{\tau+v}{\tau+u}\right)^a \frac{g(v)}{1+v/\tau} \, dv .
\end{eqnarray}
The leading term in the above integral can be further evaluated via integration by part
\begin{eqnarray}
 \int_{-\tau}^u \left(\frac{\tau+v}{\tau+u}\right)^a \frac{g(v)}{1+v/\tau} \, dv
&=& \frac{\tau}{(\tau+u)^a }\int_0^{\tau+u} w^{a-1} g(w-\tau) \, dw \, ,\nonumber\\
&=&  \frac{\tau}{(\tau+u)^a } \left(  \left[ \frac{w^a}{a} g(w-\tau) \right]_0^{\tau+u} - \int_0^{\tau+u} \frac{w^a}{a} g'(w-\tau) \, dw \right) \, ,
\end{eqnarray}
where the integral in the right-hand side is $O_{-\tau}(\tau+u)$.
Taking the limit $u \to -\tau$ in the remaining term yields the announced value 
\begin{eqnarray}
L(-\tau)
=
\lim_{u \to -\tau} \int_{-\tau}^u \left(\frac{\tau+v}{\tau+u}\right)^a \frac{g(v)}{1+v/\tau} \, dv
=
\lim_{u \to -\tau} \frac{\tau g(u)}{a} = \frac{ g(-\tau)}{ f(-\tau)} \, ,
\end{eqnarray}
showing that $L$ is continuous on $\mathbb{R}$.

$iii)$ \emph{Differentiability.} 
Let us first evaluate $L'(\tau)$ by Taylor expanding $L(u)$ in $-\tau$ to first order.
First, by repeated integration by parts, we obtain 
\begin{eqnarray}\label{eq:repIPP1}
\lefteqn{
\frac{1}{\tau} \int_0^u \frac{f(v)}{1+v/\tau} \, dv
= f(u)\ln{\left(\Big \vert1+\frac{u}{\tau} \Big \vert\right)} -} \nonumber\\
&& \hspace{85pt} f'(u) (\tau+u) \left(\ln{\left(\Big \vert1+\frac{u}{\tau}\Big \vert \right)} -1  \right) - \tau f'(0) + F(u) \, ,
\end{eqnarray}
where the last term $F(u)$ refers to the function continuously differentiable function
\begin{eqnarray}
F(u) = \int_0^u f''(v)(\tau+v)\left( \ln{\left(\Big \vert1+\frac{v}{\tau} \Big \vert \right)} -1 \right) \, dv \, .
\end{eqnarray}
Noticing that $F'(-\tau) =0$, we have $F(v)-F(u) =  o_{-\tau}(\tau+u)$ when $\vert \tau+v \vert < \vert \tau+u \vert$.
Moreover, Taylor expanding $f$ and $f'$ around $-\tau$ yields
\begin{eqnarray}
\lefteqn{
f(u)\ln{\left(\Big \vert1+\frac{u}{\tau}\Big \vert \right)} - f'(u) (\tau+u) \left(\ln{\left(\Big \vert 1+\frac{u}{\tau} \Big \vert \right)} -1  \right) = } \nonumber\\
&& \hspace{40pt} f(-\tau) \ln{\left(\Big \vert 1+\frac{u}{\tau} \Big \vert \right)} + f'(-\tau) (\tau+u) + o_{-\tau}(\tau+u) \, .
\end{eqnarray}
Thus, when $u \to -\tau$, $\vert \tau+v \vert< \vert \tau+u \vert$,
the first-order approximation to the exponent function in \eqref{eq:solAnalytical} is
\begin{eqnarray}
\hspace{5pt} -\frac{1}{\tau} \int_v^u \frac{f(v)}{1+v/\tau} \, dv
=
f(-\tau) \ln{\left(\Big \vert \frac{\tau+v}{\tau+u}\Big \vert \right)} + f'(-\tau)( v- u) +  o_{-\tau}(\tau+u) \, .
\end{eqnarray}
In turn, to first-order in $\tau+u$, we have the asymptotic behavior for $L(u)$
\begin{eqnarray}
L(u) = \int_{-\tau}^u \left(\frac{\tau+v}{\tau+u}\right)^a \big(1+f'(-\tau)( v- u)\big) \frac{g(v)}{1+v/\tau} \, dv  + o_{-\tau}(\tau+u) \, .
\end{eqnarray}
To write the above relation as an explicit linear approximation, we split the above expression
in three terms that we evaluate separately: $L(u)=A(u)+B(u)+C(u)$.
The linear approximation to the first term is obtained by repeated integration by part
\begin{eqnarray}\label{eq:repIPP2}
A(u) &=& \int_{-\tau}^u \left(\frac{\tau+v}{\tau+u}\right)^a \frac{g(v)}{1+v/\tau} \, dv  \, ,\\
&=& \frac{\tau g(u)}{a} -\frac{\tau g'(u)}{a(a +1)} (\tau+u) + o_{\tau+u}(\tau+u) \, , \nonumber\\
&=& L(-\tau) + \frac{\tau g'(-\tau)}{a+1} (\tau+u) + o_{\tau+u}(\tau+u) \, . \nonumber
\end{eqnarray}
while the linear approximations to the remaining terms only requires one integration by part:
\begin{eqnarray}
B(u)&=& \tau f'(-\tau) \int_{-\tau}^u \left(\frac{\tau+v}{\tau+u}\right)^a (\tau+u)\frac{ g(v)}{1+v/\tau} \, dv \, ,\\
&=& \frac{\tau^2f'(-\tau)g(-\tau)}{a}(\tau+u) + o_{\tau+u}(\tau+u) \, . \nonumber
\end{eqnarray}
\begin{eqnarray}
C(u)&=& \tau f'(-\tau) \int_{-\tau}^u \left(\frac{\tau+v}{\tau+u}\right)^a (\tau+v)\frac{ g(v)}{1+v/\tau} \, dv \, ,\\
&=& \frac{\tau^2f'(-\tau)g(-\tau)}{a+1}(\tau+u) + o_{\tau+u}(\tau+u) \, .\nonumber
\end{eqnarray}
Remembering that $a=\tau f(-\tau)$, we find the announced limit behavior
\begin{eqnarray}
\lim_{u \to - \tau} \frac{L(u)-L(-\tau)}{\tau+u} &=& \frac{\tau g'(-\tau)}{a+1}-\tau^2f'(-\tau)g(-\tau) \left(\frac{1}{a}-\frac{1}{a+1}\right) \, , \\
&=& \frac{a}{a+1} \left(\frac{g}{f} \right)'(-\tau) \, . \nonumber
\end{eqnarray}
Derivatives of higher order are obtained via similar, albeit intricate,
calculations evaluating the higher-order Taylor expansions of $L(u)$ around $-\tau$.
The maximum order for this expansion is determined by the number of times that
integration by part can be performed in step \eqref{eq:repIPP1} and step \eqref{eq:repIPP2}.
The maximum order is therefore $n-1$ for functions $f$ and $g$ in $C^{(n)}(\mathbb{R})$,
which implies that $L$ has a derivative of order $n-1$ in $-\tau$.
\end{proof}


\begin{remark}
Proposition \ref{prop:analysis1} actually holds for equations of the form
\begin{eqnarray}\label{eq:analysisODE2}
h(u+\tau) L'(u)  + f(u) L(u) -  g(u) = 0 \, ,
\end{eqnarray}
where $h$ is continuously differentiable with a single root: $h(0)=0$, $h'(0)>0$.
Knowing continuous differentiability, the value
\begin{eqnarray}
L'(-\tau)(u+\tau) =\frac{\left( g/f\right)'(-\tau)}{1+h'(0)/f(-\tau)} \, ,
\end{eqnarray}
directly follows from linearizing  \eqref{eq:analysisODE2} and from using $L(-\tau)=g(-\tau)/f(-\tau)$.
\end{remark}



\subsubsection{Complete monotonicity of the smooth solution}\label{sec:compmon}
The following lem\-ma will be the key to prove the complete monotonicity of the smooth solutions to the RMF \emph{ansatz}.

\begin{lemma}\label{lm:growth}
Let $f$ and $g$ be real-valued functions in $C^2(\mathbb{R})$ such that $f>0$, $g>0$,
and $f'<0$, $g'>0$ on an open interval $I$ containing $-\tau$.
Then, the unique continuous solution $L$ to  \eqref{eq:analysisODE} is strictly increasing on $I$.
\end{lemma}


\begin{proof}
If $g>0$ and $f>0$, expression \eqref{eq:solAnalytical} directly shows that $L$ remains positive on $\mathbb{R}$.
As $L$ is solution to  \eqref{eq:analysisODE2} and $f>0$ on $I$,
$L$ is increasing on $I$ if and only if $L \geq g/f$ on $(-\infty,-\tau) \cap I$ and $L(u) \leq g(u)/f(u)$ on $(-\tau,\infty) \cap I$.
Let us show that $L$ is below the curve of $g/f$ on $(-\tau,\infty) \cap I$ by contradiction. 
First, observe that by Proposition  \ref{prop:analysis1}, we know that the curve of $L$ intersects
the curve of $g/f$ in $-\tau$ with a slope $L'(-\tau)<(g/f)'(-\tau)$.
In particular, $L<g/f$ on the interval $(-\tau, -\tau+\epsilon)$ for small enough $\epsilon>0$.
Suppose there is $u$ in $I$, $u>-\tau+\epsilon$, such that $L(u)>g/f(u)$, then the set 
\begin{eqnarray}
V=\lbrace v \in I \cap (-\tau+\epsilon, +\infty ) \, \vert \, L(v)=g(v)/f(v) \rbrace
\end{eqnarray}
is non empty by continuity of $L$ and $g/f$.
Consider the first hitting time: $v_0 =\inf  V>-\tau$. 
By definition, $L$ remains below $g/f$ on $(-\tau,v_0)$ and we must have $L'(v_0) = 0$.
However, $f/g$ is a strictly increasing function when $f>0$, $g>0$, and $f'<0$, $g'>0$.
Thus, $(f/g)'(v_0) >0=L'(v_0)$ while $(f/g)(v_0) = L(v_0)$, which implies that $f/g<L$ in the left vicinity of $v_0$.
This contradicts the definition of $v_0$ as the first-hitting time.
The same argument applies on $(-\tau,\infty)$ to show that the curve of $L$ above the curve of $g/f$ on $(-\infty,-\tau)$.
\end{proof}


We are now in a position to prove a result of monotonicity for derivatives of all orders
via a simple recurrence argument, which is equivalent to the property of complete monotonicity.


\begin{proposition}\label{prop:analysis2}
Let $f$ and $g$ be real-valued functions in $C^\infty(\mathbb{R})$ such that for all $u<0$,
we have $f(u)>0$, $g(u)>0$ and $f^{(n)}(u)<0$, $g^{(n)}(u)>0$ for all $n$ in $\mathbb{N}_*$.
Then, the unique continuous solution $L$ to  \eqref{eq:analysisODE}
is such that for all $n$ in $\mathbb{N}$ and for all $u < 0$, we have $L^{(n)}(u) > 0$.
\end{proposition}


\begin{proof}

$i)$ The first step is to exhibit a system of first-order ODEs satisfied by
the $(n\!+\!1)$-th order derivatives $L^{(n+1)}$.
Proposition \ref{prop:analysis1} directly implies that the continuous solution $L$ to 
Equation \eqref{eq:analysisODE} is in $C^\infty(\mathbb{R})$ on $\mathbb{R}$ if $f$ and $g$ are in $C^\infty(\mathbb{R})$. 
Repeated differentiation of  \eqref{eq:analysisODE} on
$\mathbb{R} \setminus \lbrace -\tau \rbrace$ shows that for all $n$ in $\mathbb{N}$, the functions $L^{(n+1)}$ satisfy
\begin{eqnarray}
\left( 1+ \frac{u}{\tau}\right) L^{(n+1)}(u)  + f_n(u)  L^{(n)}(u) -  g_n(u) = 0 \, ,
\end{eqnarray}
where we have $f_n = n/\tau+f$ and where the function $g_n$ is defined by recurrence as
\begin{eqnarray}
g_n(u) = g'_{n-1}(u) - f'(u) L^{(n-1)}(u) \, , \quad \mathrm{with} \quad g_0 = g(u) \, . 
\end{eqnarray}
Proceeding inductively, we obtain an explicit expression for $g_n$:
\begin{eqnarray}
g_n(u)=g^{(n)}(u)-\sum_{k=0}^{n-1} \frac{d^k}{du^k}\Big( f'(u) L^{(n-1-k)}(u) \Big) \, ,
\end{eqnarray}
which can by further simplified via the Leibniz formula and the hockey-stick identity 
\begin{eqnarray}\label{eq:itergn}
g_n(u)&=&g^{(n)}(u)-\sum_{k=0}^{n-1} \sum_{l=0}^{k}  {k\choose{l}} f^{(l+1)}(u) L^{(n-1-l)}(u) \, , \\
&=& g^{(n)}(u)-\sum_{l=0}^{n-1}  {n\choose{l+1}} f^{(l+1)}(u) L^{(n-1-l)}(u)  \, .
\end{eqnarray}

$ii)$ The proof then proceeds by recurrence on the order of the derivative.
We know that the unique continuous solution to \eqref{eq:analysisODE} is a positive function: $L>0$.
Suppose that $L^{(k)}>0$, for $1 \leq k \leq n$, i.e., that the functions $L^{(k)}$,
$0 \leq k \leq n-1$, are positive increasing functions on $\mathbb{R}_-$.
Formula \eqref{eq:itergn}  shows that $g_n$ is also positive increasing: $g_n>0$ and $g_n'>0$.
Then, observing that $f_n$ and $g_n$ in  \eqref{eq:analysisODE2} satisfy
the hypotheses of Lemma \ref{lm:growth} with $I=(-\infty,0)$, we conclude that $L^{(n)}$
is positive increasing on $(-\infty,0)$, i.e. $L^{(n+1)}>0$.
By recurrence, we deduce that derivatives of all order are positive: $L^{(n)}>0$ on $(-\infty,0)$ for all $n$ in $\mathbb{N}$.
\end{proof}


\subsubsection{Existence of a solution to the RMF \emph{ansatz}} \label{sec:relaxsol}


The proof of \Cref{th:relaxingNeuron} mirrors the argument of the proof of \Cref{th:counting},
except that one has to check that $i)$ the smooth solutions of the RMF \emph{ansatz}
are indeed MGFs and $ii)$ that the self-consistency equations for the mean neuronal intensities admit at least one solution.

\begin{proof}[Proof of \Cref{th:relaxingNeuron}]
$i)$ \emph{Necessary conditions on the mean intensities.}
Given positive mean intensities $\beta_j>0$, $1 \leq j \leq K$, each equation
of the system \eqref{eq:sysLaplace} can be written under the same form as
 $\eqref{eq:solAnalytical}$ by introducing the functions 
\begin{eqnarray}
f_i(u) = - \frac{ub_i}{\tau_i}  +\sum_{j \neq i}\left( 1-e^{u \mu_{ij}}\right) \beta_j  \quad \mathrm{and} \quad g_i(u)=\beta_i e^{u r_i} \, ,
\end{eqnarray}
which belong to $C^\infty(\mathbb{R})$ with $f_i(-\tau_i)>0$.
Thus, by Proposition \ref{prop:analysis1}, each equation of the system \eqref{eq:sysLaplace}
admits the unique continuous solution on $\mathbb{R}$

\begin{eqnarray}
L_i(u) = \int_{-\tau_i}^u e^{-\int_v^u \frac{f_i(w)}{1+\tau_i/w} \, dw}  \frac{g_i(v)}{1+v/\tau_i} \, dv \, , \quad 1 \leq j \leq K \, ,
\label{eq:usedeq}
\end{eqnarray}
which also belong to $C^\infty(\mathbb{R})$.
Moreover, the functions $f_i$ and $g_i$ are such that for all $u<0$, $f_i(u) > 0$,
$g_i(u)>0$, $f_i^{(n)}(u)< 0$ and $g_i^{(n)}(u)>0$ if $\beta_j>0$ for $1 \leq j \leq K$.
Thus, by Proposition \ref{prop:analysis2}, we deduce that the functions $L_i$, $1 \leq j \leq K$,
have strictly positive derivative at all order in $(-\infty, 0)$.
Together, the above properties state that the functions defined by
$u \mapsto L_i(-u)$ are completely monotone function on $(0, \infty)$ \cite{Feller2:1971}.
By Bernstein's theorem on completely monotone functions, $u \mapsto L_i(-u)$
is the Laplace transform of a positive measure $m_i$ defined on the Borel sets of $\mathbb{R}_+$, that is:
\begin{eqnarray}
L_i(-u) = \int_0^\infty e^{-ut} dm_i(t) \, .
\end{eqnarray}
In particular, the functions $L_i$ are MGFs if and only if the measures $m_i$ are probability measure.
This is equivalent to imposing that $L_i(0)=1$, $1\leq i \leq K$,
which gives the announced system of equations \eqref{eq:sysMI} for the mean intensities $\beta_j$.
Operating the change of variables $y=\tau_i \ln{\left(1+v/\tau_i\right)}$ and
$x=\tau_i \ln{\left(1+w/\tau_i\right)}$ yields the integral expression
 \begin{eqnarray}
\lefteqn{
L_i(u) =  \beta_i \times } \nonumber\\
&& \hspace{25pt}  \int_{-\infty}^u \exp{\left(\int_y^u b_i\left(e^{\frac{x}{\tau_i}} -1 \right)  +\sum_{j \neq i} \left( e^{\tau_i \mu_{ij}\left(e^{\frac{x}{\tau_i}} -1 \right)}-1\right) \beta_j \, dx \right)}  e^{\tau_i r_i \left(e^{\frac{y}{\tau_i}} -1 \right)}\, dy\, , 
\end{eqnarray}
which reduces to  \eqref{eq:solMGF} after evaluating the integral exponent,
therefore justifying the announced system of equations \eqref{eq:sysMI} for the mean intensities $\beta_j$.

$ii)$ \emph{Existence of mean intensities solutions.}
In order to show the existence of solutions to the system of equations \eqref{eq:sysMI},
let us consider the map $\bm{F}:\mathbbm{R}_+^K \rightarrow \mathbbm{R}_+^K$ whose components are defined by
\begin{eqnarray}
\hspace{20pt} F_i(\bm{\beta}) = \left(\int_{-\infty}^0 \exp{\left( \left[h_i(x)-\sum_{j \neq i} \beta_j h_{ij}(x) \right]^0_v+l_i(v)\right)}   \, dv \right)^{-1}  , \:\: 1 \leq i\leq K \, .
\end{eqnarray}
Given $\bm{\beta}_0$ in the positive orthant, iterating the map $\bm{F}$ specifies a sequence 
$\lbrace \bm{\beta}_n \rbrace_{n \in \mathbb{N}}$, $\bm{\beta}_{n} = \bm{F}^n(\bm{\beta}_0)$,
whose finite accumulation points are solutions to \eqref{eq:sysMI}.
To establish that such accumulation points exist, it is enough to show that the positive sequence
$\lbrace \bm{\beta}_n \rbrace_{n \in \mathbb{N}}$ is bounded.
Given $\bm{\beta}_0$ in the positive orthant, we show the boundedness of
$\lbrace \bm{\beta}_n \rbrace_{n \in \mathbb{N}}$ by exhibiting a dominating convergent sequence.
The first step is to observe that for $t \leq 0$, we have:
\begin{eqnarray}
h_i(0)-h_i(x)+l_i(x) = \tau_i(r_i-b_i)\left(e^{x/\tau_i}-1\right)+b_i x \geq \max(b_i,r_i) x \, ,
\end{eqnarray}
and consequently, we have
\begin{eqnarray}
F_i(\bm{\beta}) \leq   \left(\int_{-\infty}^0 \exp{\left(\max(b_i,r_i) v-\sum_{j \neq i} \beta_j \Big[ h_{ij}(x) \Big]_v^0 \right)}   \, dv \right)^{-1} \stackrel{\rm def}{=} \tilde{F}_i(\bm{\beta}) \, .
\end{eqnarray}
Because of the convexity of the exponential function, the newly introduced function $\tilde{F}_i$
turns out to be an increasing function of the relaxation time $\tau_i$, so that we have
\begin{eqnarray}
\hspace{15pt} F_i(\bm{\beta}) \leq \lim_{\tau_i \to \infty} \tilde{F}_i(\bm{\beta} )=  \left(\int_{-\infty}^0 \exp{\left( r_i t +\sum_{j \neq i} \beta_j \left( \frac{1-e^{t \mu_{ij}}}{\mu_{ij}} + t\right) \right)}   \, dt \right)^{-1} \, .
\end{eqnarray}
Observing that $\lim_{\tau_i \to \infty} \tilde{F}_i(\bm{\beta} )$ 
is also an increasing function of the parameters $\mu_{ij}$ and $r_i$, we further have
\begin{eqnarray}
F_i(\bm{\beta})  \leq \left(\int_{-\infty}^0 \exp{\left( r t + \left( \frac{1-e^{t \mu}}{\mu} + t\right)\sum_{j \neq i} \beta_j \right)}   \, dt \right)^{-1} \stackrel{\rm def}{=} G_i(\bm{\beta} )  \, , 
\end{eqnarray}
where $r=\max_i r_i$ and $\mu = \max_{i,j} \mu_{ij}$.
As expected, evaluating the integral in the above expression for
$\bm{\beta}=\beta\bm{1}$ yields the equation associated to the counting-neuron model
with interaction weight $\mu$ and base intensity equal to the reset value $r$:
\begin{eqnarray}
G_i(\beta \bm{1})= \frac{\mu c^{c+x}e^{-c}}{\gamma(c+x,c)} \stackrel{\rm def}{=} g(\beta)   \, , \quad \mathrm{with} \quad c=\frac{(K-1) \beta}{\mu} \quad \mathrm{and} \quad x=\frac{r}{\mu} \, .
\end{eqnarray}
Given $\bm{\beta}_0$ in the positive orthant, posit 
$\bm{\beta}_0' = (\max_i \beta_{0,i}) \bm{1}$ and consider the two sequences
$\lbrace \bm{\beta}_n \rbrace_{n \in \mathbb{N}}$ and $\lbrace \bm{\beta}_n '\rbrace_{n \in \mathbb{N}}$
obtained by iterating the maps $\bm{F}$ and $\bm{G}$ on $\bm{\beta}_0$ and $\bm{\beta}_0'$, respectively:
$\bm{\beta}_{n} = \bm{F}^n(\bm{\beta}_0)$ and $\bm{\beta}_{n} '= \bm{G}^n(\bm{\beta}_0')$.
If $\bm{\beta}_n \leq \bm{\beta}_n'$, then
$\bm{\beta}_{n+1}=\bm{F}(\bm{\beta}_n) \leq \bm{F}(\bm{\beta}_n') \leq \bm{G}(\bm{\beta}_n')  = \bm{\beta}_{n+1}'$, 
where we have used the fact that for all $1 \leq i \leq K$, $F_i$ is increasing with respect to $\beta_j$, $1 \leq j \leq K$:
\begin{eqnarray}
\partial_{\beta_j} F_i(\bm{\beta}) = -\frac{\int_{-\infty}^0 h_{ij}(t) e^{\left(h_i(t)+\sum_{j \neq i} \beta_j h_{ij}(t) \right)}   \, dt}{F_i(\bm{\beta})^2} \geq 0 \, .
\end{eqnarray}
Thus, as $\bm{\beta}_0 \leq \bm{\beta}'_0$ by construction,
the sequence $\lbrace \bm{\beta}_n' \rbrace_{n \in \mathbb{N}}$ dominates
$\lbrace \bm{\beta}_n \rbrace_{n \in \mathbb{N}}$ with respect to the product order in $\mathbb{R}_K$.
It remains to show to $\lbrace \bm{\beta}_n' \rbrace_{n \in \mathbb{N}}$ is convergent,
which is equivalent to show that the one dimensional sequence
$\lbrace \beta'_n \rbrace_{n \in \mathbb{N}}$, $\beta_n'= g^n(\max_i \beta_{0,i})$, is convergent.
To justify this point, it is enough to check that the sequence
$\lbrace \beta'_n \rbrace_{n \in \mathbb{N}}$ is bounded, as Lemma \ref{lm:counting}
shows that there is a unique fixed point solution to $\beta= g(\beta) =  \mu c^ae^{-c} / \gamma(a,c)$.
Introducing the rescaled sequence $\lbrace c_n \rbrace_{n \in \mathbb{N}}$ defined by
$c_n = (K-1)\beta'_n/\mu$, notice that $c_{n+1} =h(c_n)$ with
\begin{eqnarray}
h(c) = (K-1)\frac{ c^{c+x}e^{-c}}{\gamma(c+x,c)}  \, .
\end{eqnarray}
From the power expansion of the incomplete gamma function, we have
\begin{eqnarray}
h(c) &=&  (K-1)\left( \sum_{n \geq 0} \frac{c^n}{(x+c)(x+c+1)\ldots (x+c+n)}\right)^{-1} \, \nonumber\\
 &\leq&  (K-1)\left( \sum_{n = 0}^{K-1} \frac{c^n}{(x+c)(x+c+1)\ldots (x+c+n)}\right)^{-1}  = \frac{K-1}{K} c + o_\infty(c)\, .
\end{eqnarray}
showing that $h(c)<c$ for large enough $c$.
This implies that  $\lbrace c_n \rbrace_{n \in \mathbb{N}}$ is a bounded sequence,
and so is $\lbrace \beta'_n \rbrace_{n \in \mathbb{N}}$.
\end{proof}


\section{Future directions}\label{sec:future}

Our results were obtained and discussed for
purely excitatory LGL networks and limited to first-order RMF \emph{ansatz}.
We would like to stress that, in principle, our approach to reduce RMF \emph{ans\"atze}
to a set of self-consistency equations---founded on imposing the condition of analyticity on the solutions to
the \emph{ans\"atze}---can be generalized to models including inhibition and higher-order statistics. 

In the context of second-order RMF, the RCP can be applied to the joint 
MGF of pairs of neurons rather than single neurons. Our replica framework can be extended to simplify the representation of the point processes that feed
this pair through some appropriate extension of the Poisson Hypothesis. The interactions
between the two neurons of the pair are however described in an exact way.
An important complication of our replica approach for higher order is that  the RMF \emph{ansatz} consists in a system of PDEs rather than a system of ODEs.
However, the PDEs associated with the RCP for second-order RMF model can be solved using 
singularity-analysis techniques generalizing those described in this work.
This line of thought is essential to represent, e.g., the wave phenomena present in cyclic networks, which limits the applicability of first-order RMF networks.
Second-order RMF networks are expected to bring essential new features absent from order one.
They are most probably the least complex networks within the RMF class allowing
one to capture correlation effects. They also seem to provide the 
least complex networks that are not fundamentally time irreversible, i.e., with a positive production of entropy.

Another important extension is to account for networks supporting both excitatory and inhibitory interactions within our RMF framework. 
Including inhibitory interactions within a point-process framework requires to consider nonlinear models of synaptic integration, whereby stochastic intensities can remain non-negative in spite of inhibitory inputs.
There are several possible nonlinear models which are biophysically relevant, each yielding distinct functional characterizations of their RMF stationary state.
Considering these nonlinear RMF networks in toy models shows that singularity-analysis techniques are still applicable to networks with mixed excitation and inhibition.
However, the presence of inhibition fundamentally alters the nature of the singularity featuring in the non-physical solutions to the RMF \emph{ansatz}.
Generalizing our analysis to singularities that are more involved than infinite discontinuities is the key challenge to include inhibition within our framework.
Importantly, we have numerical evidence that networks with inhibition have RMF versions that admit several stable solutions.
We intend to utilize these multistable RMF networks to probe the metastable behavior of the finite-size networks that share the same neural basic structure.

The above computational questions will be explored in companion papers.
A more fundamental question remains to prove the propagation of chaos in finite-replica models, which is supported by simulations and is the central conjecture of this work.

\section*{Acknowledgments} T.T. was supported by the Alfred P. Sloan Research Fellowship FG-2017-9554.
 F.B. was supported by an award from the Simons Foundation (\#197982). Both awards are
to the University of Texas at Austin.

\bibliographystyle{siamplain}
\bibliography{siam_bib}

\begin{thebibliography}{10}

\bibitem{Abbott:1993}
{\sc L.~F. Abbott and C.~van Vreeswijk}, {\em Asynchronous states in networks
  of pulse-coupled oscillators}, Phys. Rev. E, 48 (1993), pp.~1483--1490,
  \url{https://doi.org/10.1103/PhysRevE.48.1483},
  \url{https://link.aps.org/doi/10.1103/PhysRevE.48.1483}.

\bibitem{Abeles:1995}
{\sc M.~Abeles, H.~Bergman, I.~Gat, I.~Meilijson, E.~Seidemann, N.~Tishby, and
  E.~Vaadia}, {\em Cortical activity flips among quasi-stationary states},
  Proceedings of the National Academy of Sciences, 92 (1995), pp.~8616--8620,
  \url{https://doi.org/10.1073/pnas.92.19.8616},
  \url{http://www.pnas.org/content/92/19/8616},
  \url{https://arxiv.org/abs/http://www.pnas.org/content/92/19/8616.full.pdf}.

\bibitem{Amari:1975aa}
{\sc S.-I. Amari}, {\em Homogeneous nets of neuron-like elements}, Biological
  Cybernetics, 17 (1975), pp.~211--220,
  \url{https://doi.org/10.1007/BF00339367},
  \url{https://doi.org/10.1007/BF00339367}.

\bibitem{Amari:1977aa}
{\sc S.-I. Amari}, {\em Dynamics of pattern formation in lateral-inhibition
  type neural fields}, Biological Cybernetics, 27 (1977), pp.~77--87,
  \url{https://doi.org/10.1007/BF00337259},
  \url{https://doi.org/10.1007/BF00337259}.

\bibitem{Amit:1985}
{\sc D.~J. Amit, H.~Gutfreund, and H.~Sompolinsky}, {\em Storing infinite
  numbers of patterns in a spin-glass model of neural networks}, Phys. Rev.
  Lett., 55 (1985), pp.~1530--1533,
  \url{https://doi.org/10.1103/PhysRevLett.55.1530},
  \url{https://link.aps.org/doi/10.1103/PhysRevLett.55.1530}.

\bibitem{BremaudBaccelli:2003}
{\sc F.~Baccelli and P.~Br\'emaud}, {\em Elements of queueing theory}, vol.~26
  of Applications of Mathematics (New York), Springer-Verlag, Berlin,
  second~ed., 2003, \url{https://doi.org/10.1007/978-3-662-11657-9},
  \url{https://doi.org/10.1007/978-3-662-11657-9}.
\newblock Palm martingale calculus and stochastic recurrences, Stochastic
  Modelling and Applied Probability.

\bibitem{Baccelli:2017aa}
{\sc F.~Baccelli, F.~Mathieu, and I.~Norros}, {\em Mutual service processes in
  euclidean spaces: existence and ergodicity}, Queueing Systems, 86 (2017),
  pp.~95--140, \url{https://doi.org/10.1007/s11134-017-9524-3},
  \url{https://doi.org/10.1007/s11134-017-9524-3}.

\bibitem{Baccelli:2018aa}
{\sc F.~Baccelli, A.~Rybko, S.~Shlosman, and A.~Vladimirov}, {\em Metastability
  of queuing networks with mobile servers}, Journal of Statistical Physics,
  (2018), \url{https://doi.org/10.1007/s10955-018-2023-z},
  \url{https://doi.org/10.1007/s10955-018-2023-z}.

\bibitem{Baladron:2012aa}
{\sc J.~Baladron, D.~Fasoli, O.~Faugeras, and J.~Touboul}, {\em Mean-field
  description and propagation of chaos in networks of hodgkin-huxley and
  fitzhugh-nagumo neurons}, The Journal of Mathematical Neuroscience, 2 (2012),
  p.~10, \url{https://doi.org/10.1186/2190-8567-2-10},
  \url{https://doi.org/10.1186/2190-8567-2-10}.

\bibitem{Benaim:121369}
{\sc M.~Benaim and J.-Y. Le~Boudec}, {\em A class of mean field interaction
  models for computer and communication systems}, Performance Evaluation, 65
  (2008), pp.~823--838.

\bibitem{bramson:2011}
{\sc M.~Bramson}, {\em Stability of join the shortest queue networks}, Ann.
  Appl. Probab., 21 (2011), pp.~1568--1625,
  \url{https://doi.org/10.1214/10-AAP726},
  \url{https://doi.org/10.1214/10-AAP726}.

\bibitem{Bressloff:2009aa}
{\sc P.~Bressloff}, {\em Stochastic neural field theory and the system-size
  expansion}, SIAM Journal on Applied Mathematics, 70 (2009), pp.~1488--1521,
  \url{https://doi.org/10.1137/090756971},
  \url{https://doi.org/10.1137/090756971}.

\bibitem{Brunel:2000aa}
{\sc N.~Brunel}, {\em Dynamics of sparsely connected networks of excitatory and
  inhibitory spiking neurons}, Journal of Computational Neuroscience, 8 (2000),
  pp.~183--208, \url{https://doi.org/10.1023/A:1008925309027},
  \url{https://doi.org/10.1023/A:1008925309027}.

\bibitem{Brunel:1999aa}
{\sc N.~Brunel and V.~Hakim}, {\em Fast global oscillations in networks of
  integrate-and-fire neurons with low firing rates}, Neural Computation, 11
  (1999), pp.~1621--1671, \url{https://doi.org/10.1162/089976699300016179},
  \url{https://doi.org/10.1162/089976699300016179}.

\bibitem{Buice:2013}
{\sc M.~A. Buice and C.~C. Chow}, {\em Dynamic finite size effects in spiking
  neural networks}, PLOS Computational Biology, 9 (2013), pp.~1--21,
  \url{https://doi.org/10.1371/journal.pcbi.1002872},
  \url{https://doi.org/10.1371/journal.pcbi.1002872}.

\bibitem{Buice:2007}
{\sc M.~A. Buice and J.~D. Cowan}, {\em Field-theoretic approach to fluctuation
  effects in neural networks}, Phys. Rev. E, 75 (2007), p.~051919,
  \url{https://doi.org/10.1103/PhysRevE.75.051919},
  \url{https://link.aps.org/doi/10.1103/PhysRevE.75.051919}.

\bibitem{Buice:2009aa}
{\sc M.~A. Buice, J.~D. Cowan, and C.~C. Chow}, {\em Systematic fluctuation
  expansion for neural network activity equations}, Neural Computation, 22
  (2009), pp.~377--426, \url{https://doi.org/10.1162/neco.2009.02-09-960},
  \url{https://doi.org/10.1162/neco.2009.02-09-960}.

\bibitem{Churchland:2010aa}
{\sc M.~M. Churchland, B.~M. Yu, J.~P. Cunningham, L.~P. Sugrue, M.~R. Cohen,
  G.~S. Corrado, W.~T. Newsome, A.~M. Clark, P.~Hosseini, B.~B. Scott, D.~C.
  Bradley, M.~A. Smith, A.~Kohn, J.~A. Movshon, K.~M. Armstrong, T.~Moore,
  S.~W. Chang, L.~H. Snyder, S.~G. Lisberger, N.~J. Priebe, I.~M. Finn,
  D.~Ferster, S.~I. Ryu, G.~Santhanam, M.~Sahani, and K.~V. Shenoy}, {\em
  Stimulus onset quenches neural variability: a widespread cortical
  phenomenon}, Nature Neuroscience, 13 (2010), pp.~369 EP --,
  \url{http://dx.doi.org/10.1038/nn.2501}.

\bibitem{Daley1}
{\sc D.~J. Daley and D.~Vere-Jones}, {\em An introduction to the theory of
  point processes. {V}ol. {I}}, Probability and its Applications (New York),
  Springer-Verlag, New York, second~ed., 2003.
\newblock Elementary theory and methods.

\bibitem{Daley2}
{\sc D.~J. Daley and D.~Vere-Jones}, {\em An introduction to the theory of
  point processes. {V}ol. {II}}, Probability and its Applications (New York),
  Springer, New York, second~ed., 2008,
  \url{https://doi.org/10.1007/978-0-387-49835-5},
  \url{https://doi.org/10.1007/978-0-387-49835-5}.
\newblock General theory and structure.

\bibitem{DayanAbbott}
{\sc P.~Dayan and L.~F. Abbott}, {\em Theoretical neuroscience}, Computational
  Neuroscience, MIT Press, Cambridge, MA, 2001.
\newblock Computational and mathematical modeling of neural systems.

\bibitem{DeMasi:2015}
{\sc A.~De~Masi, A.~Galves, E.~L{\"o}cherbach, and E.~Presutti}, {\em
  Hydrodynamic limit for interacting neurons}, Journal of Statistical Physics,
  158 (2015), pp.~866--902, \url{https://doi.org/10.1007/s10955-014-1145-1},
  \url{https://doi.org/10.1007/s10955-014-1145-1}.

\bibitem{Delattre:2016}
{\sc S.~Delattre, N.~Fournier, and M.~Hoffmann}, {\em Hawkes processes on large
  networks}, Ann. Appl. Probab., 26 (2016), pp.~216--261,
  \url{https://doi.org/10.1214/14-AAP1089},
  \url{https://doi.org/10.1214/14-AAP1089}.

\bibitem{Doiron:2016aa}
{\sc B.~Doiron, A.~Litwin-Kumar, R.~Rosenbaum, G.~K. Ocker, and K.~Josi{\'c}},
  {\em The mechanics of state-dependent neural correlations}, Nature
  Neuroscience, 19 (2016), pp.~383 EP --,
  \url{http://dx.doi.org/10.1038/nn.4242}.

\bibitem{Ecker:2016}
{\sc A.~S. Ecker, G.~H. Denfield, M.~Bethge, and A.~S. Tolias}, {\em On the
  structure of neuronal population activity under fluctuations in attentional
  state}, Journal of Neuroscience, 36 (2016), pp.~1775--1789,
  \url{https://doi.org/10.1523/JNEUROSCI.2044-15.2016},
  \url{http://www.jneurosci.org/content/36/5/1775},
  \url{https://arxiv.org/abs/http://www.jneurosci.org/content/36/5/1775.full.pdf}.

\bibitem{Faugeras:2009}
{\sc O.~Faugeras, J.~Touboul, and B.~Cessac}, {\em A constructive mean-field
  analysis of multi population neural networks with random synaptic weights and
  stochastic inputs}, Frontiers in Computational Neuroscience, 3 (2009), p.~1,
  \url{https://doi.org/10.3389/neuro.10.001.2009},
  \url{https://www.frontiersin.org/article/10.3389/neuro.10.001.2009}.

\bibitem{Feller2:1971}
{\sc W.~Feller}, {\em An introduction to probability theory and its
  applications. {V}ol. {II}}, Second edition, John Wiley \& Sons, Inc., New
  York-London-Sydney, 1971.

\bibitem{Galves:2013}
{\sc A.~Galves and E.~L{\"o}cherbach}, {\em Infinite systems of interacting
  chains with memory of variable length---a stochastic model for biological
  neural nets}, Journal of Statistical Physics, 151 (2013), pp.~896--921,
  \url{https://doi.org/10.1007/s10955-013-0733-9},
  \url{https://doi.org/10.1007/s10955-013-0733-9}.

\bibitem{Gardner:1988}
{\sc E.~Gardner}, {\em The space of interactions in neural network models},
  Journal of Physics A: Mathematical and General, 21 (1988), p.~257,
  \url{http://stacks.iop.org/0305-4470/21/i=1/a=030}.

\bibitem{Gillespie:1977}
{\sc D.~T. Gillespie}, {\em Exact stochastic simulation of coupled chemical
  reactions}, The Journal of Physical Chemistry, 81 (1977), pp.~2340--2361,
  \url{https://doi.org/10.1021/j100540a008},
  \url{http://pubs.acs.org/doi/abs/10.1021/j100540a008},
  \url{https://arxiv.org/abs/http://pubs.acs.org/doi/pdf/10.1021/j100540a008}.

\bibitem{Goris:2014aa}
{\sc R.~L.~T. Goris, J.~A. Movshon, and E.~P. Simoncelli}, {\em Partitioning
  neuronal variability}, Nature Neuroscience, 17 (2014), pp.~858 EP --,
  \url{http://dx.doi.org/10.1038/nn.3711}.

\bibitem{Hodara:2018}
{\sc P.~Hodara, N.~Krell, and E.~L\"ocherbach}, {\em Non-parametric estimation
  of the spiking rate in systems of interacting neurons}, Stat. Inference
  Stoch. Process., 21 (2018), pp.~81--111,
  \url{https://doi.org/10.1007/s11203-016-9150-4},
  \url{https://doi.org/10.1007/s11203-016-9150-4}.

\bibitem{Hodara:2017}
{\sc P.~Hodara and E.~L\"ocherbach}, {\em Hawkes processes with variable length
  memory and an infinite number of components}, Adv. in Appl. Probab., 49
  (2017), pp.~84--107, \url{https://doi.org/10.1017/apr.2016.80},
  \url{https://doi.org/10.1017/apr.2016.80}.

\bibitem{Jacod:1975}
{\sc J.~Jacod}, {\em Multivariate point processes: predictable projection,
  {R}adon-{N}ikod\'ym derivatives, representation of martingales}, Z.
  Wahrscheinlichkeitstheorie und Verw. Gebiete, 31 (1974/75), pp.~235--253,
  \url{https://doi.org/10.1007/BF00536010},
  \url{https://doi.org/10.1007/BF00536010}.

\bibitem{Kerstan:1964}
{\sc J.~Kerstan}, {\em Teilprozesse {P}oissonscher {P}rozesse}, in Trans.
  {T}hird {P}rague {C}onf. {I}nformation {T}heory, {S}tatist. {D}ecision
  {F}unctions, {R}andom {P}rocesses ({L}iblice, 1962), Publ. House Czech. Acad.
  Sci., Prague, 1964, pp.~377--403.

\bibitem{Lin:2015aa}
{\sc I.-C. Lin, M.~Okun, M.~Carandini, and K.~D. Harris}, {\em The nature of
  shared cortical variability}, Neuron, 87 (2015), pp.~644--656,
  \url{https://doi.org/https://doi.org/10.1016/j.neuron.2015.06.035},
  \url{http://www.sciencedirect.com/science/article/pii/S089662731500598X}.

\bibitem{Locherbach:2018}
{\sc E.~L\"ocherbach}, {\em Absolute continuity of the invariant measure in
  piecewise deterministic {M}arkov {P}rocesses having degenerate jumps},
  Stochastic Process. Appl., 128 (2018), pp.~1797--1829,
  \url{https://doi.org/10.1016/j.spa.2017.08.011},
  \url{https://doi.org/10.1016/j.spa.2017.08.011}.

\bibitem{Mathes:1964}
{\sc K.~Matthes}, {\em Zur {T}heorie der {B}edienungsprozesse}, in Trans.
  {T}hird {P}rague {C}onf. {I}nformation {T}heory, {S}tatist. {D}ecision
  {F}unctions, {R}andom {P}rocesses ({L}iblice, 1962), Publ. House Czech. Acad.
  Sci., Prague, 1964, pp.~513--528.

\bibitem{McKean:1966}
{\sc H.~P. McKean}, {\em A class of markov processes associated with nonlinear
  parabolic equations}, Proceedings of the National Academy of Sciences, 56
  (1966), pp.~1907--1911, \url{https://doi.org/10.1073/pnas.56.6.1907},
  \url{http://www.pnas.org/content/56/6/1907},
  \url{https://arxiv.org/abs/http://www.pnas.org/content/56/6/1907.full.pdf}.

\bibitem{Mecke:1967}
{\sc J.~Mecke}, {\em Station\"{a}re zuf\"{a}llige {M}asse auf lokalkompakten
  {A}belschen {G}ruppen}, Z. Wahrscheinlichkeitstheorie und Verw. Gebiete, 9
  (1967), pp.~36--58, \url{https://doi.org/10.1007/BF00535466},
  \url{https://doi.org/10.1007/BF00535466}.

\bibitem{Melamed:1979}
{\sc B.~Melamed}, {\em Characterizations of poisson traffic streams in jackson
  queueing networks}, Advances in Applied Probability, 11 (1979), pp.~422--438,
  \url{http://www.jstor.org/stable/1426847}.

\bibitem{MeynTweedie:2009}
{\sc S.~Meyn and R.~L. Tweedie}, {\em Markov chains and stochastic stability},
  Cambridge University Press, Cambridge, second~ed., 2009,
  \url{https://doi.org/10.1017/CBO9780511626630},
  \url{https://doi.org/10.1017/CBO9780511626630}.
\newblock With a prologue by Peter W. Glynn.

\bibitem{MeynTweedie:1993}
{\sc S.~P. Meyn and R.~L. Tweedie}, {\em Stability of markovian processes iii:
  Foster--lyapunov criteria for continuous-time processes}, Advances in Applied
  Probability, 25 (1993), pp.~518--548, \url{https://doi.org/10.2307/1427522}.

\bibitem{Mezard:1987aa}
{\sc M.~M\'{e}zard, G.~Parisi, and M.~A. Virasoro}, {\em Spin glass theory and
  beyond}, vol.~9 of World Scientific Lecture Notes in Physics, World
  Scientific Publishing Co., Inc., Teaneck, NJ, 1987.

\bibitem{Ocker:2017aa}
{\sc G.~K. Ocker, K.~Josi{\'c}, E.~Shea-Brown, and M.~A. Buice}, {\em Linking
  structure and activity in nonlinear spiking networks}, PLOS Computational
  Biology, 13 (2017), pp.~e1005583--,
  \url{https://doi.org/10.1371/journal.pcbi.1005583}.

\bibitem{Pillow:2008aa}
{\sc J.~W. Pillow, J.~Shlens, L.~Paninski, A.~Sher, A.~M. Litke, E.~J.
  Chichilnisky, and E.~P. Simoncelli}, {\em Spatio-temporal correlations and
  visual signalling in a complete neuronal population}, Nature, 454 (2008),
  pp.~995 EP --, \url{http://dx.doi.org/10.1038/nature07140}.

\bibitem{Bialek:1999}
{\sc F.~Rieke, D.~Warland, R.~de~Ruyter~van Steveninck, and W.~Bialek}, {\em
  Spikes}, A Bradford Book, MIT Press, Cambridge, MA, 1999.
\newblock Exploring the neural code, Computational Neuroscience.

\bibitem{Robert:2016}
{\sc P.~Robert and J.~Touboul}, {\em On the dynamics of random neuronal
  networks}, J. Stat. Phys., 165 (2016), pp.~545--584,
  \url{https://doi.org/10.1007/s10955-016-1622-9},
  \url{https://doi.org/10.1007/s10955-016-1622-9}.

\bibitem{RybShlosI}
{\sc A.~Rybko and S.~Shlosman}, {\em Poisson hypothesis for information
  networks. {I}}, Mosc. Math. J., 5 (2005), pp.~679--704, 744.

\bibitem{RybShlosII}
{\sc A.~Rybko and S.~Shlosman}, {\em Poisson hypothesis for information
  networks. {II}}, Mosc. Math. J., 5 (2005), pp.~927--959, 974.

\bibitem{Rybko:2009aa}
{\sc A.~Rybko, S.~Shlosman, and A.~Vladimirov}, {\em Spontaneous resonances and
  the coherent states of the queuing networks}, Journal of Statistical Physics,
  134 (2009), pp.~67--104, \url{https://doi.org/10.1007/s10955-008-9658-0},
  \url{https://doi.org/10.1007/s10955-008-9658-0}.

\bibitem{SB17}
{\sc A.~Sankararaman and F.~Baccelli}, {\em Spatial birth-death wireless
  networks}, IEEE Tr. Information Theory,  (2017), pp.~3964--3982.

\bibitem{Sompolinski:1988}
{\sc H.~Sompolinsky, A.~Crisanti, and H.~J. Sommers}, {\em Chaos in random
  neural networks}, Phys. Rev. Lett., 61 (1988), pp.~259--262,
  \url{https://doi.org/10.1103/PhysRevLett.61.259},
  \url{https://link.aps.org/doi/10.1103/PhysRevLett.61.259}.

\bibitem{Sznitman:1989}
{\sc A.-S. Sznitman}, {\em Topics in propagation of chaos}, in \'Ecole
  d'\'Et\'e de {P}robabilit\'es de {S}aint-{F}lour {XIX}---1989, vol.~1464 of
  Lecture Notes in Math., Springer, Berlin, 1991, pp.~165--251,
  \url{https://doi.org/10.1007/BFb0085169},
  \url{https://doi.org/10.1007/BFb0085169}.

\bibitem{Takacs:1962aa}
{\sc L.~Tak{\'a}cs}, {\em Introduction to the theory of queues}, Oxford
  University Press, 1962.

\bibitem{Tognoli:2014aa}
{\sc E.~Tognoli and J.~A.~S. Kelso}, {\em The metastable brain}, Neuron, 81
  (2014), pp.~35--48, \url{https://doi.org/10.1016/j.neuron.2013.12.022},
  \url{https://doi.org/10.1016/j.neuron.2013.12.022}.

\bibitem{Touboul:2012aa}
{\sc J.~Touboul, G.~Hermann, and O.~Faugeras}, {\em Noise-induced behaviors in
  neural mean field dynamics}, SIAM Journal on Applied Dynamical Systems, 11
  (2012), pp.~49--81, \url{https://doi.org/10.1137/110832392},
  \url{https://doi.org/10.1137/110832392}.

\bibitem{Touboul:2011}
{\sc J.~D. Touboul and G.~B. Ermentrout}, {\em Finite-size and
  correlation-induced effects in mean-field dynamics}, J. Comput. Neurosci., 31
  (2011), pp.~453--484, \url{https://doi.org/10.1007/s10827-011-0320-5},
  \url{https://doi.org/10.1007/s10827-011-0320-5}.

\bibitem{Emery:2004}
{\sc W.~Truccolo, U.~T. Eden, M.~R. Fellows, J.~P. Donoghue, and E.~N. Brown},
  {\em A point process framework for relating neural spiking activity to
  spiking history, neural ensemble, and extrinsic covariate effects}, Journal
  of Neurophysiology, 93 (2005), pp.~1074--1089,
  \url{https://doi.org/10.1152/jn.00697.2004},
  \url{https://doi.org/10.1152/jn.00697.2004},
  \url{https://arxiv.org/abs/https://doi.org/10.1152/jn.00697.2004}.
\newblock PMID: 15356183.

\bibitem{Vvedenskaya:1996}
{\sc N.~D. {Vvedenskaya}, R.~L. {Dobrushin}, and F.~I. {Karpelevich}}, {\em
  {Queueing system with selection of the shortest of two queues: An asymptotic
  approach.}}, {Probl. Inf. Transm.}, 32 (1996), pp.~15--27.

\end{thebibliography}
\end{document}